\documentclass[hidelinks,onefignum,onetabnum]{siamart251216}

\usepackage{lipsum}
\usepackage{amsfonts}
\usepackage{graphicx}
\usepackage{epstopdf}
\usepackage{amsfonts}
\usepackage{amssymb}
\usepackage{enumitem}
\usepackage{algorithmic}
\ifpdf
  \DeclareGraphicsExtensions{.eps,.pdf,.png,.jpg}
\else
  \DeclareGraphicsExtensions{.eps}
\fi

\newsiamremark{remark}{Remark}
\newsiamremark{hypothesis}{Hypothesis}
\crefname{hypothesis}{Hypothesis}{Hypotheses}
\newsiamthm{claim}{Claim}
\newsiamremark{fact}{Fact}
\crefname{fact}{Fact}{Facts}

\headers{Tensor Envelopes in the Nonconvex Setting}{A. Kabgani and M. Ahookhosh}

\title{Tensor Envelopes in the Nonconvex Setting\thanks{
\funding{The Research Foundation Flanders (FWO) research project G081222N and UA BOF DocPRO4 projects with ID 46929 and 48996 partially supported the paper's authors.}}}

\author{Alireza Kabgani\thanks{Department of Mathematics, University of Antwerp, Antwerp, Belgium.
  (\email{alireza.kabgani@uantwerp.be}, \email{masoud.ahookhosh@uantwerp.be}).}
\and Masoud Ahookhosh\footnotemark[2]}

\usepackage{amsopn}

\newsiamremark{example}{Example}

\newcommand\R{\mathbb{R}}
\newcommand{\E}{\mathbb{E}}
\newcommand{\N}{\mathbb{N}}
\newcommand\Rinf{\overline{\mathbb{R}}}
\newcommand\dom[1]{ \bs{{\rm dom}}(#1)} 
\newcommand\dist{ \bs{{\rm dist}}} 
\newcommand\gf{\varphi} 
\newcommand\gh{\psi} 
\newcommand\fgam[3]{#1_{#3}^{#2}}

\newcommand\prox[3]{ \bs{{\rm prox}}_{#2#1}^{#3}}

\newcommand\ov[1]{\overline{#1}}
\newcommand\mb{\mathbf{B}}
\newcommand\bs[1]{\boldsymbol{#1}}

\newcommand{\Model}{\mathcal{Q}}
\newcommand{\Fix}{\operatorname{\textbf{Fix}}}
\newcommand\argmint[1]{\mathop{\bs{\arg\min}}\limits_{#1}}
\newcommand\Nz{\mathbb{N}_0}

\newcommand{\Om}[2]{\Omega_{#1,#2}}
\newcommand\Tprox[3]{\bs T_{#2}^{#3}#1}
\newcommand\maj[1]{\mathcal{M}^{#1}}
\newcommand\argmin[1]{\bs{\arg\min}_{#1}}

\newcommand{\M}{\mathcal{M}}
\newcommand{\pos}[1]{\left(#1\right)_{+}}

\newcommand{\ip}[2]{\left\langle #1,#2\right\rangle}
\newcommand{\norm}[1]{\left\lVert#1\right\rVert}
\newcommand{\abs}[1]{\left\lvert#1\right\rvert}

\newcommand{\crit}{\operatorname{\textbf{crit}}}
\newcommand\inft[1]{\mathop{\bs{\inf}}\limits_{#1}}
\newcommand\supt[1]{\mathop{\bs{\sup}}\limits_{#1}}
\newcommand\liminft[1]{\mathop{\bs{\liminf}}\limits_{#1}}
\newcommand\limsupt[1]{\mathop{\bs{\limsup}}\limits_{#1}}

\newcommand\maxt[1]{\mathop{\bs{\max}}\limits_{#1}}
\newcommand\mint[1]{\mathop{\bs{\min}}\limits_{#1}}
\let\oldliminf\liminf
\renewcommand{\liminf}{\bs{\oldliminf}}
\let\oldlimsup\limsup
\renewcommand{\limsup}{\bs{\oldlimsup}}
\let\oldmax\max
\renewcommand{\max}{\bs{\oldmax}}
\let\oldmin\min
\renewcommand{\min}{\bs{\oldmin}}
\let\oldsup\sup
\renewcommand{\sup}{\bs{\oldsup}}
\let\oldinf\inf
\renewcommand{\inf}{\bs{\oldinf}}
\let\oldlim\lim
\renewcommand{\lim}{\bs{\oldlim}}

\ifpdf
\hypersetup{
  pdftitle={Tensor Envelopes in the Nonconvex Setting},
  pdfauthor={A. Kabgani and M. Ahookhosh}
}
\fi

\begin{document}

\maketitle

\begin{abstract}
We introduce a tensor envelope (TEE) for composite objectives $\gf=f+g$, built from a $q$th-order Taylor model of the smooth term $f$ regularized by a $p$th power with $p>q$. To construct it, we first establish a tensor descent lemma and characterize its behavior in the three regimes
$q<p<q+1$, $p=q+1$, and $p>q+1$. Under suitable prox-boundedness and parameter conditions, the associated tensor operator (TOP) has nonempty compact values, and TEE is finite-valued and continuous, preserves global minimizers, and inherits the level-boundedness of $\gf$. We further show that TEE is locally Lipschitz continuous and directionally differentiable, characterize its Fr\'echet differentiability, and prove that it is $C^1$ whenever TOP is single-valued. As such, TEE preserves the essential optimization landscape properties of $\gf$ while enjoying a richer
analytical structure, which facilitates the design and analysis of tensor-based optimization methods. Finally, TEE serves as a Lyapunov function for the exact tensor iteration: every cluster point is critical, and, for bounded trajectories with $q<p\le q+1$, an $\epsilon$-stationary
point is reached within $\mathcal{O}(\epsilon^{-p/(p-1)})$ iterations. To our knowledge, both the tensor descent lemma and this complexity bound are the first of their kind obtained without assuming global Lipschitz or H\"older continuity of the $q$th-order derivative of $f$.

\end{abstract}

\begin{keywords}
nonsmooth and nonconvex optimization, majorization minimization, tensor envelope, tensor descent lemma,  generic tensor method, complexity analysis.
\end{keywords}

\begin{MSCcodes}
49J52, 65K10, 90C26, 49J53
\end{MSCcodes}

\section{Introduction}
Composite minimization models play a fundamental role in modern optimization, with a broad range of applications involving nonsmooth and nonconvex objectives. Classical methods for nonsmooth optimization, such as subgradient and bundle methods, can be sensitive to parameter tuning, and subgradient methods in particular are known to converge slowly. To tame nonsmoothness, a large class of proximal envelopes has been introduced (e.g., Moreau \cite{Moreau65}, forward-backward \cite{Stella17,themelis2019acceleration,Themelis18}, and Douglas-Rachford \cite{themelis2020douglas} envelopes, as well as their Bregman and higher-order variants \cite{Ahookhosh21,Kabganidiff,Kabgani24itsopt,kabgani2025fundamental}) to stand as exact and smooth surrogates for the original objective under suitable conditions. In particular, while retaining the minimizers and stationary points of the objective is essential from an optimization perspective, the smoothing property of proximal envelopes offers the possibility of designing efficient first- and second-order methods; cf. \cite{Ahookhosh21,Kabgani25itsdeal,Kabgani24itsopt}. Furthermore, they can stand as essential tools to further boost proximal splitting algorithms via special line search on these envelopes; see, e.g., \cite{Ahookhosh21,kabgani2025fb,Kabgani25itsdeal,Kabgani24itsopt,Stella17,themelis2019acceleration,themelis2020douglas,Themelis18}. 

Let us consider the composite minimization problem
\begin{equation}\label{eq:intro_problem}
\mint{x\in\E}\ \gf(x):=f(x)+g(x),
\end{equation}
where $\E$ is a finite-dimensional real vector space, $f:\E\to\R$ is sufficiently smooth, and $g:\E\to\Rinf$ is proper and lower semicontinuous (lsc). Since $f$ is smooth, it is natural to consider the model-based subproblem
\begin{equation}\label{eq:opt_model}
\mint{y\in\E}\left\{\M_{f,x}(y)+g(y)+\frac{1}{p\gamma}\|y-x\|^p\right\},
\end{equation}
where $\M_{f,x}$ denotes a smooth model of $f$ around $x$, the term $\frac{1}{p}\|\cdot-x\|^p$ serves as a regularization term with $p>1$, and $\gamma>0$ is a regularization parameter. Depending on the order of smoothness information available for $f$, one may consider its zero-, first-, second-, or higher-order Taylor approximation at $x$, leading to a broad range of models that include proximal methods \cite{Ahookhosh2025,Ahookhosh24,Kabgani24itsopt,martinet1970breve, martinet1972determination, Nesterov2023a,rockafellar1976monotone2}, proximal-gradient methods \cite{beck2017first,fukushima1981generalized}, cubic regularization methods \cite{cartis2011adaptive,CartisGouldToint2011ARC,doikov2021minimizing,NesterovPolyak2006}, and tensor methods \cite{birgin2017worst,CartisGouldToint2020Sharp,doikov2022local,doikov2020inexact,martinez2017high,Nesterov2021,nesterov2021superfast}. With a suitable envelope in hand, these methods can be further accelerated through an appropriate line search, as demonstrated in the ItsOPT framework \cite{Kabgani24itsopt}.

Recent results on high-order tensor methods confirm the effectiveness of employing high-order Taylor models in the composite setting when the function $f$ is sufficiently smooth; see, e.g., \cite{GrapigliaNesterov2020HolderTensor,grapiglia2022tensor,NabouNecoara2024Composite,NabouNecoara2026,NecoaraLupu2025,nesterov2021superfast}. These considerations, together with the boosted versions of these methods developed within the ItsOPT framework (cf. \cite{Kabgani24itsopt}), motivate the development and study of proximal envelopes in high-order settings, which exploit higher-order local information of the original objective. In detail, let $\Omega^f_{x,q}$ denote the $q$th-order Taylor expansion of $f$ at a point $x$. 
For an integer $q\ge 1$, a regularization power $p>q$, and $\gamma>0$, we consider the {\it high-order tensor operator (TOP)} given by
\begin{equation}\label{eq:tpo}
    \Tprox{\gf}{\gamma}{p,q}(x):=\argmint{y\in\E} \left\{\Omega^f_{x,q}(y)+g(y)+\frac{1}{p\gamma}\|y-x\|^p\right\},
\end{equation}
and define its corresponding {\it tensor envelope (TEE)} as
\begin{equation}\label{eq:value}
    \fgam{\gf}{p,q}{\gamma}(x):=\inft{y\in\E} \left\{\Omega^f_{x,q}(y)+g(y)+\frac{1}{p\gamma}\|y-x\|^p\right\}.
\end{equation}
If the one-sided Taylor remainder satisfies
\[
f(y)\leq \Omega^f_{x,q}(y)+\frac{L}{p}\|y-x\|^p \quad \forall x,y\in\E,
\]
then the model in \eqref{eq:value} majorizes $\gf$ whenever $\gamma^{-1}\ge L$. A central question is therefore to understand when such a one-sided majorization property holds and how it depends on the Taylor order $q$ and the regularization power $p$.
For $q=1$, the above construction reduces to the high-order forward--backward mapping and envelope studied in \cite{kabgani2025fb,kabgani2025fundamental}. For higher-order Taylor models, however, despite the extensive development of tensor models and high-order proximal methods, the fundamental properties of the resulting tensor envelope remain largely unexplored. In particular, its
majorization, well-definedness, continuity, differentiability, and stationarity properties, and hence its optimization landscape, require a separate analysis, which is the main topic of the current study.

\subsection{{\bf Contribution}}\label{sec:contribution}

Our contributions are twofold:
\begin{description}[wide, labelwidth=!, labelindent=0pt]
\item[(i)] \textbf{Tensor descent lemma.}
We introduce a one-sided majorization condition for the $q$th-order Taylor model of $f$ regularized by a power $p>q$ and study the resulting class $\maj{q}(p,L)$. We characterize this class in the three regimes $q<p<q+1$, $p=q+1$, and $p>q+1$. For $q<p\le q+1$, we relate the majorization property to a one-sided H\"older-type variation of $D^qf$. At the endpoint $p=q+1$, we uncover a parity-dependent structure: for even $q$, the condition is equivalent to global Lipschitz continuity of $D^qf$, whereas for odd $q$ it controls only the positive part of the $(q+1)$st-order directional derivative. For $p>q+1$, we show that any finite majorization constant necessarily collapses to zero. Examples show that the resulting class can extend beyond functions whose $q$th derivative is globally H\"older continuous.

\item[(ii)]\textbf{Tensor envelope and generic tensor method.}
We introduce the tensor operator (TOP) and its associated tensor envelope (TEE) and establish their basic variational properties, including nonemptiness and compactness of TOP and finite-valuedness, continuity, and level-boundedness properties of TEE. Under suitable assumptions, TEE and the original objective have the same global minimizers. We further derive local Lipschitz and directional differentiability properties, characterize Fr\'echet differentiability, and show that single-valuedness of TOP yields $C^1$ regularity. For the canonical choice $p=q+1$, we characterize directional stationary points of TEE through fixed points of TOP and relate them to critical points of the original objective. Finally, using TEE as a Lyapunov function for the exact tensor iteration, we establish asymptotic regularity and criticality of cluster points. For bounded trajectories with $q<p\le q+1$, we also obtain the first-order stationarity complexity $\mathcal O\left(\epsilon^{-p/(p-1)}\right)$, without imposing global H\"older or Lipschitz continuity on $D^qf$.
To the best of our knowledge, this is the first study of the tensor envelope (beyond the first-order forward-backward \cite{kabgani2025fundamental,Stella17,Themelis18}) and systematically investigate its fundamental and differentiability properties, thereby opening a new avenue for the development of efficient algorithms for nonconvex composite optimization; see, e.g., \cite{Kabgani25itsdeal,Kabgani24itsopt}.
\end{description}

\subsection{{\bf Organization}}\label{sec:organization}
The remainder of the paper is organized as follows.
Section~\ref{sec:prelim} introduces the notation and preliminary results.
Section~\ref{sec:ten_maj} studies the tensor descent lemma and its different 
regularization regimes. Section~\ref{sec:Ten_env} develops the tensor operator and tensor
envelope, including their basic properties, together with stationarity complexity bounds for generic tensor method.
Section~\ref{sec:conclusion} concludes the paper.

\section{Preliminaries and notation} \label{sec:prelim}
Let $\E$ be a finite-dimensional real vector space with dual $\E^*$, and let $\ip{s}{x}$ denote the dual pairing between $s\in\E^*$ and $x\in\E$. We set $\N_0:=\N\cup\{0\}$.
For $C\subseteq\E$, let $\delta_C$ denote its indicator function.
For $a\in\R$, set $\pos{a}:=\max\{a,0\}$. Let $B:\E\to\E^*$ be self-adjoint and positive definite, $B=B^*\succ0$, and define $\norm{x}:=\ip{Bx}{x}^{1/2}$ for $x\in \E$ and $\norm{s}_*:=\ip{s}{B^{-1}s}^{1/2}$ for $s\in \E^*$.
For $p>1$ and $u\in \E$, $\nabla \frac{1}{p}\norm{u}^p=\norm{u}^{p-2}Bu$,
where the right-hand side is understood as zero at $u=0$.
We adopt the conventions $(+\infty)^{-1}:=0$ and $0^{-1}:=+\infty$.
For $\gh:\E\to\Rinf:=\R\cup\{+\infty\}$, we set
$\dom \gh:=\{x\in\E: \gh(x)<+\infty\}$. The function $\gh$ is coercive if $\bs\lim_{\Vert x \Vert \to +\infty} \gh(x) = +\infty$.
We use the beta--gamma identity
\begin{equation}\label{eq:beta-gamma}
\int_0^1 t^{a-1}(1-t)^{b-1}\,dt = \frac{\Gamma(a)\Gamma(b)}{\Gamma(a+b)}, \quad a,b>0,
\end{equation}
where $\Gamma(m)=(m-1)!$ for $m\in\N$.
For $j\geq 1$ and $f\in C^j(\E)$, the $j$th derivative $D^j f(x):\E^j\to\R$ is a symmetric $j$-linear form. By $D^j f(x)[h]^j:=D^j f(x)[h,\ldots,h]$, its operator norm is
\[
\norm{D^j f(x)}:=\max\left\{\left|D^j f(x)[h_1,\ldots,h_j]\right|:\norm{h_i}\le 1,
\ i=1,\ldots,j\right\}.
\]
For every symmetric $j$-linear form $A$ on $\E$, 
\cite[Eq.~(1.2)]{Nesterov2021},
\begin{equation}\label{eq:symmetric_tensor_norm}
\|A\| = \maxt{\|h\|\le1}|A[h]^j|.
\end{equation}
For $f\in C^q(\E)$, the order-$q$ Taylor polynomial of $f$ at $x$ evaluated at $y$ is denoted by
\[
\Om{x}{q}^f(y):=f(x)+\sum_{j=1}^q \frac{1}{j!}D^j f(x)[y-x]^j.
\]
For $\nu\in(0,1]$, $D^q f$ is globally $\nu$-H\"older continuous with constant $H\geq0$ if
\[
\norm{D^q f(y)-D^q f(x)}\le H\norm{y-x}^{\nu} \quad \forall x,y\in \E.
\]
The case $\nu=1$ corresponds to global Lipschitz continuity.

For $x\in\dom \gh$, the \textit{Fr\'{e}chet/regular} and \textit{Mordukhovich/limiting subdifferentials} of $\gh$ at $x$ \cite{Mordukhovich2018,Rockafellar09} are, respectively, 
\[
\widehat\partial \gh(x):=\left\{v\in\E^*:\liminft{\substack{y\to x\\ y\ne x}}\frac{\gh(y)-\gh(x)
-\langle v,y-x\rangle}{\|y-x\|}\ge 0\right\},
\]
and
\[
\partial \gh(x):=\left\{v\in\E^*:\exists\, x^k\to x,\ \gh(x^k)\to \gh(x),\ v^k\to v,
\ v^k\in\widehat\partial \gh(x^k)\right\}.
\]
For $x\notin\dom\gh$, we set $\widehat{\partial}\gh(x)=\partial\gh(x):=\emptyset$.
For $\varphi=f+g$, define
\[
\crit(\gf):=\{x\in\dom g:0\in \nabla f(x)+\partial g(x)\}.
\]
We recall the high-order proximal operator and the high-order Moreau envelope.
\begin{definition}[High-order proximal operator and Moreau envelope]
Let $p>1$, $\gamma>0$, and let $\psi:\E\to\Rinf$ be proper. The high-order proximal operator (\textit{HOPE}) is
\begin{equation}\label{eq:Hiorder-Moreau prox}
\prox{\psi}{\gamma}{p}(x):=\argmint{y\in\E}\left\{\psi(y)+\frac{1}{p\gamma}\|x-y\|^p\right\}.
\tag{HOPE}
\end{equation}
The high-order Moreau envelope (\textit{HOME}) is
\begin{equation}\label{eq:Hiorder-Moreau env}
\fgam{\psi}{p}{\gamma}(x):=\inft{y\in\E}\left\{\psi(y)+\frac{1}{p\gamma}\|x-y\|^p\right\}.
\tag{HOME}
\end{equation}
\end{definition}
We recall the following notion of high-order prox-boundedness.
\begin{definition}[High-order prox-boundedness]\label{def:s-prox-bounded}\cite[Definition~3.3]{Kabgani24itsopt}
A proper function $\psi:\E\to\Rinf$ is high-order prox-bounded of order $p>1$ if there are $\gamma>0$ and $x\in\E$ such that $\fgam{\psi}{p}{\gamma}(x)>-\infty$. The supremum of all such parameters is denoted by $\gamma^{\psi,p}$ and is called the high-order prox-boundedness threshold of $\psi$.
\end{definition}

\begin{remark}\label{rem:prox-bounded}
If $\psi$ is convex or bounded below, then $\gamma^{\psi,p}=+\infty$.
\end{remark}
We conclude this section by recalling the basic properties of HOME and HOPE.
\begin{fact}[HOME and HOPE properties]\label{fact:level-bound+locally uniform}\cite{Kabgani24itsopt}
Let $p>1$ and let $\psi:\E\to\Rinf$ be proper, lsc, and high-order prox-bounded with threshold $\gamma^{\psi,p}>0$. Then, for each $\gamma\in(0,\gamma^{\psi,p})$, the function
$(y,x)\mapsto \psi(y)+\frac{1}{p\gamma}\|y-x\|^p$
is lsc and level-bounded in $y$ locally uniformly in $x$; see \cite[Definition~1.16]{Rockafellar09}. For every $x\in\E$, the set $\prox{\psi}{\gamma}{p}(x)$
is nonempty and compact, and $\fgam{\psi}{p}{\gamma}$ is finite-valued.
The mapping $(x,\gamma)\mapsto\fgam{\psi}{p}{\gamma}(x)$ is continuous on $\E\times(0,\gamma^{\psi,p})$. If $x^k\to x$, $\gamma_k\to\gamma\in(0,\gamma^{\psi,p})$, and $y^k\in\prox{\psi}{\gamma_k}{p}(x^k)$, 
then $\{y^k\}_{k\in\Nz}$ is bounded and every cluster point belongs to $\prox{\psi}{\gamma}{p}(x)$.
\end{fact}

\section{Tensor majorant meets high-order regularization} \label{sec:ten_maj} 

High-order optimization methods commonly combine Taylor approximations with power regularization to construct tractable models of the objective; see, e.g., \cite{martinez2017high,NecoaraLupu2025}. In majorization-based methods, the model is required to upper-bound the objective while agreeing
with it at the current point, so that minimizing the model naturally yields descent.

For $\gf=f+g$, we keep the nonsmooth term $g$ unchanged. Since the Taylor model is exact at the base point, majorization reduces to controlling the Taylor remainder of $f$. Standard H\"older assumptions bound the absolute Taylor error $\left|f(y)-\Omega^f_{x,q}(y)\right|$, whereas, for
majorization, it suffices to control the amount by which $f(y)$ exceeds its Taylor model, namely,
\[
\pos{f(y)-\Omega^f_{x,q}(y)}.
\]
The purpose of this section is to characterize when this quantity admits a global bound of order $\norm{y-x}^p$ and to understand how this property depends on the Taylor order $q$ and the regularization power $p$. In particular, we distinguish the regimes $q<p<q+1$, $p=q+1$, and $p>q+1$,
which exhibit substantially different behavior. Together, they provide the majorization framework for the tensor operator and envelope studied in Section~\ref{sec:Ten_env}.

\begin{definition}[tensor descent lemma]\label{def:maj}
Let $q\geq 1$ be an integer, $p>1$, and $L\geq 0$. A function $f\in C^q(\E)$ belongs to $\maj{q}(p,L)$ if
\begin{equation}\label{eq:maj_ineq}
    f(y)\le \Omega^f_{x,q}(y)+\frac{L}{p}\|y-x\|^p \quad \forall x,y\in\E,
\end{equation}
which is also called {\it $(p,q)$-tensor majorant}. The smallest admissible majorization constant is
\begin{equation*}
L_{p,q}(f):=\supt{x\ne y}\frac{p\bigl(f(y)-\Omega^f_{x,q}(y)\bigr)_+}{\|y-x\|^p}\in[0,+\infty],
\end{equation*}
i.e., $f\in\maj{q}(p,L)$ if and only if $L_{p,q}(f)\le L$.
\end{definition}

For $q=1$, Definition~\ref{def:maj} reduces to the first-order high-order majorant studied in \cite{kabgani2025fb}.
Although Definition~\ref{def:maj} allows $p\le q$, we do not study this range further. Indeed, for every $f\in C^q(\E)$ and fixed $x\in\E$, $f(x+h)-\Omega^f_{x,q}(x+h)=o(\|h\|^q)$, and consequently, if $p\le q$, we have
\[
\frac{\pos{f(x+h)-\Omega^f_{x,q}(x+h)}}{\|h\|^p}\to 0,
\quad\text{as }h\to 0.
\]
Hence, in this range the local Taylor behavior does not determine whether $L_{p,q}(f)$ is finite. We therefore focus on $p>q$, where comparison with the first omitted Taylor order $q+1$ leads to the three distinct regimes $q<p<q+1$, $p=q+1$, and $p>q+1$.

We repeatedly use the following integral form of the Taylor remainder.
\begin{fact}[Integral Taylor remainder]\label{lem:dm_remainder}
If $f\in C^q(\E)$, for every $x,h\in\E$, then
\begin{equation}\label{eq:integral_remainder}
f(x+h)-\Omega_{x,q}^f(x+h) = \frac{1}{(q-1)!}\int_0^1(1-t)^{q-1} \big(D^qf(x+th)-D^qf(x)\big)[h]^q\,dt.
\end{equation}
\end{fact}

The classical H\"older assumption on the highest Taylor derivative provides a sufficient condition for membership in the tensor-majorant class. In particular, \cite[Lemma~2.1]{CartisGouldToint2020Sharp}
gives the following implication.

\begin{fact}[H\"older continuity and tensor majorant]\label{lem:holder_sufficient}
Let $\nu\in(0,1]$ and suppose that $D^qf$ is globally $\nu$-H\"older continuous with constant $H\geq 0$. Then
$f\in\maj{q}(q+\nu,L)$ with $L=H\Gamma(\nu+1)/\Gamma(q+\nu)$.
In particular, if $D^q f$ is globally Lipschitz with constant $H$, then $f\in\maj{q}(q+1,H/q!)$.
\end{fact}

Note that the global H\"older continuity of $D^qf$ is sufficient for membership in the tensor-majorant class, but, as the examples below show, it is not necessary.
Next, we record elementary calculus rules for the class.
\begin{proposition}[Basic calculus rules]\label{prop:basic}
Let $q\ge1$ and $p>1$. Then,
\begin{enumerate}[label=(\textbf{\alph*}), font=\normalfont\bfseries, leftmargin=0.7cm]
\item\label{prop:basic:a} if $a\ge0$ and $f\in\maj{q}(p,L)$, then $af\in\maj{q}(p,aL)$;
\item\label{prop:basic:b} if $f_i\in\maj{q}(p,L_i)$ for $i=1,\ldots,N$, then
$\sum_{i=1}^Nf_i\in\maj{q}\left(p,\sum_{i=1}^NL_i\right)$;

\item\label{prop:basic:c} if $P$ is a polynomial of degree at most $q$, then $f\in\maj{q}(p,L)$ if and only if 
$f+P\in\maj{q}(p,L)$. In particular, $L_{p,q}(f+P)=L_{p,q}(f)$;
\item\label{prop:basic:d} let $\E_0$ be a finite-dimensional Euclidean space, let $A:\E_0\to\E$ be linear, and let $b\in\E$. If $f\in\maj{q}(p,L)$, then $u\mapsto f(Au+b)$ belongs to $\maj{q}(p,L\|A\|^p)$ on $\E_0$.
\end{enumerate}
\end{proposition}

\subsection{Structure of the tensor-majorant class}
In this subsection, we analyze how the tensor-majorant condition depends on the relation between the Taylor order $q$ and the regularization power $p$.

\paragraph{The case $q<p\leq q+1$}

Let $\nu\in(0,1]$, so that $p=q+\nu$.
For $f\in C^q(\E)$, define
\[
H^+_{q,\nu}(f):= \supt{x\ne y}\frac{\big((D^qf(y)-D^qf(x))[y-x]^q\big)_+}{\|y-x\|^{q+\nu}}.
\]
The following result characterizes the tensor-majorant class when $q<p\leq q+1$.
\begin{theorem}[One-sided remainder characterization: $q<p\leq q+1$]
\label{thm:fractional-characterization}
Let $\nu\in(0,1]$ and $f\in C^q(\E)$. The following assertions hold.
\begin{enumerate}[label=(\textbf{\alph*}), font=\normalfont\bfseries, leftmargin=0.7cm]
\item \label{thm:fractional-characterization:a} If $H_{q,\nu}^+(f)<+\infty$, then
\begin{equation}\label{eq:th:fra-char}
f\in \maj{q}\left(q+\nu,\frac{\Gamma(\nu+1)}{\Gamma(q+\nu)}H_{q,\nu}^+(f)\right);
\end{equation}

\item \label{thm:fractional-characterization:b} If $f\in\maj{q}(q+\nu,L)$ for some $L<+\infty$, then $H_{q,\nu}^+(f)\leq\frac{q!}{\nu}L$.

\item \label{thm:fractional-characterization:c} $H_{q,\nu}^+(f)<+\infty$ if and only if $f\in\maj{q}(q+\nu,L)$ for some $L<+\infty$. Moreover,
\begin{equation}\label{eq:fractional-equivalence}
\frac{\nu}{q!}H_{q,\nu}^+(f)\leq L_{q+\nu,q}(f)\leq \frac{\Gamma(\nu+1)}{\Gamma(q+\nu)}H_{q,\nu}^+(f),
\end{equation}
where at the endpoint $\nu=1$, $L_{q+1,q}(f) = \frac{1}{q!}H_{q,1}^+(f)$.

\item \label{thm:fractional-characterization:d}
If, in addition, $f\in C^{q+1}(\E)$, then
\begin{equation}\label{eq:endpoint-pointwise}
H_{q,1}^+(f)=\supt{\substack{x\in\E\\\norm{h}=1}}\pos{D^{q+1}f(x)[h]^{q+1}}.
\end{equation}
Moreover, $f\in\maj{q}(q+1,L)$ if and only if
\begin{equation}\label{eq:endpoint-characterization}
    D^{q+1}f(x)[h]^{q+1} \leq q!L\norm{h}^{q+1} \quad \forall x,h\in\E.
\end{equation}
\end{enumerate}
\end{theorem}
\begin{proof}
\ref{thm:fractional-characterization:a} Fix $x\in\E$, $h\neq 0$, and $t\in(0,1]$. By the definition of $H_{q,\nu}^+(f)$, we get
\[
\pos{\big(D^qf(x+th)-D^qf(x)\big)[h]^q}\leq H_{q,\nu}^+(f)t^\nu\norm{h}^{q+\nu}.
\]
Together with the integral Taylor remainder formula, this implies
\begin{align*}
\pos{f(x+h)-\Omega^f_{x,q}(x+h)}
&\leq \frac{H_{q,\nu}^+(f)}{(q-1)!}\int_0^1(1-t)^{q-1}t^\nu\,dt\, \norm{h}^{q+\nu}
\\
&=\frac{\Gamma(\nu+1)}{\Gamma(q+\nu+1)} H_{q,\nu}^+(f)\norm{h}^{q+\nu},
\end{align*}
which is \eqref{eq:th:fra-char}.
\\
\ref{thm:fractional-characterization:b}
Let $f\in\maj{q}(q+\nu,L)$ for some $L<+\infty$. Fix $z\in\E$ and $v\in\E$ with $\norm{v}=1$, and define
$a(t):=D^qf(z+tv)[v]^q$.
Since $f\in C^q(\E)$, $a$ is continuous. Let $s<t$ and set $r:=t-s$. Applying \eqref{eq:maj_ineq} with $x=z+sv$ and $y=z+tv$, and using \eqref{eq:integral_remainder} imply
\begin{equation}\label{eq:weighted-increment}
\int_0^1(1-\theta)^{q-1}\big(a(s+\theta r)-a(s)\big)\,d\theta \leq \frac{(q-1)!L}{q+\nu}r^\nu.
\end{equation}
We now claim that, for $s<t$,
\begin{equation}\label{eq:one-sided-holder-line}
a(t)-a(s) \leq \frac{q!}{\nu}L(t-s)^\nu.
\end{equation}
Suppose, by contradiction, that this fails. Consequently, there exist $s<t$ and $\lambda>q!L/\nu$ such that $a(t)-a(s)>\lambda(t-s)^\nu$. Let us define
$\Psi(u,w) := a(w)-a(u)-\lambda(w-u)^\nu$
on the compact set $\Delta:=\{(u,w):s\leq u\leq w\leq t\}$. Since $\Psi(s,t)>0$, $\Psi$ attains a positive maximum at some $(\ov u,\ov w)\in\Delta$. Necessarily $\ov u<\ov w$.
Set $\rho:=\ov w-\ov u>0$. For every $\theta\in[0,1]$, the maximality of $(\ov u,\ov w)$ ensures
$\Psi(\ov u+\theta\rho,\ov w) \leq \Psi(\ov u,\ov w)$,
consequently, $a(\ov u+\theta\rho)-a(\ov u) \geq \lambda\rho^\nu \big(1-(1-\theta)^\nu\big)$.
Applying \eqref{eq:weighted-increment} to the interval $[\ov u,\ov w]$ yields
\[
\lambda\rho^\nu \int_0^1(1-\theta)^{q-1}\big(1-(1-\theta)^\nu\big)\,d\theta \leq \frac{(q-1)!L}{q+\nu}\rho^\nu.
\]
Since $\int_0^1(1-\theta)^{q-1}\big(1-(1-\theta)^\nu\big)\,d\theta = \frac1q-\frac1{q+\nu} = \frac{\nu}{q(q+\nu)}$, we have $\lambda \leq \frac{q!}{\nu}L$, contradicting the choice of $\lambda$, i.e., \eqref{eq:one-sided-holder-line} holds. For $x\neq y$, setting $r:=\norm{y-x}$ and $v:=\frac{y-x}{r}$ and using \eqref{eq:one-sided-holder-line} with $z=x$, $s=0$, and $t=r$, we get $\big(D^qf(y)-D^qf(x)\big)[v]^q \leq \frac{q!}{\nu}Lr^\nu$.
Multiplying by $r^q$, the positive part satisfies
\[
\pos{\big(D^qf(y)-D^qf(x)\big)[y-x]^q}\leq \frac{q!}{\nu}L\norm{y-x}^{q+\nu}.
\]
Taking the supremum over $x\neq y$ proves $H_{q,\nu}^+(f)\leq \frac{q!}{\nu}L$.
\\
\ref{thm:fractional-characterization:c} The equivalence follows from Assertions~\ref{thm:fractional-characterization:a} and~\ref{thm:fractional-characterization:b}.
Assertion~\ref{thm:fractional-characterization:a} and the definition of $L_{q+\nu,q}(f)$ produce
$L_{q+\nu,q}(f) \leq \frac{\Gamma(\nu+1)}{\Gamma(q+\nu)} H_{q,\nu}^+(f)$.
On the other hand, if $L_{q+\nu,q}(f)<+\infty$, then $f\in\maj{q}(q+\nu,L_{q+\nu,q}(f))$, and Assertion~\ref{thm:fractional-characterization:b} gives
$H_{q,\nu}^+(f) \leq \frac{q!}{\nu}L_{q+\nu,q}(f)$.
Together, these inequalities prove \eqref{eq:fractional-equivalence}. For $\nu=1$, the two bounds in \eqref{eq:fractional-equivalence} coincide because $\Gamma(2)=1$ and $\Gamma(q+1)=q!$. Hence
$L_{q+1,q}(f) = \frac{1}{q!}H_{q,1}^+(f)$.
\\
\ref{thm:fractional-characterization:d}
Suppose now that $f\in C^{q+1}(\E)$, and set
\[
S:=\supt{\substack{x\in\E\\\norm{h}=1}}\pos{D^{q+1}f(x)[h]^{q+1}}.
\]
For $x\neq y$, let $d:=y-x$. By the fundamental theorem of calculus, it is clear that
\[
\big(D^qf(y)-D^qf(x)\big)[d]^q=\int_0^1 D^{q+1}f(x+td)[d]^{q+1}\,dt.
\]
i.e.,
$\pos{\big(D^qf(y)-D^qf(x)\big)[d]^q} \leq S\norm{d}^{q+1}$,
and hence $H_{q,1}^+(f)\leq S$.
Conversely, let us fix $x\in\E$ and $\norm{h}=1$. For $t>0$, the definition of $H_{q,1}^+(f)$ gives
\[
\frac{\pos{\big(D^qf(x+th)-D^qf(x)\big)[h]^q}}{t}\leq H_{q,1}^+(f).
\]
Letting $t\downarrow0$ yields
$\pos{D^{q+1}f(x)[h]^{q+1}} \leq H_{q,1}^+(f)$.
Taking the supremum over $x\in\E$ and $\norm{h}=1$ gives $S\leq H_{q,1}^+(f)$, i.e.,
$H_{q,1}^+(f)=S$, which proves \eqref{eq:endpoint-pointwise}. Combining this identity with $L_{q+1,q}(f) = \frac{1}{q!}H_{q,1}^+(f)$ and Definition~\ref{def:maj} gives
\eqref{eq:endpoint-characterization}.
\end{proof}
Indeed, for $q<p\le q+1$, the finiteness of the tensor-majorant constant is equivalent to the finiteness of a one-sided H\"older-type modulus of $D^qf$, rather than to the global H\"older continuity of $D^qf$ in the operator norm.
The next example shows that the tensor-majorant condition may hold even when $D^qf$ has no global H\"older modulus.

\begin{example}[Eighth-order Landau-type potentials]\label{ex:landau_application}
Consider the scalar eighth-order Landau-type potential
\begin{equation}\label{eq:landau_general}
\Phi(u) = P_3(u) +\frac{\beta_4}{4}u^4 +\frac{\beta_6}{6}u^6+\frac{\beta_8}{8}u^8,
\end{equation}
where $P_3$ is a polynomial of degree at most three, $\beta_4>0$, $\beta_6<0$, and $\beta_8>0$. 
Scalar eighth-order Landau models with quartic, sextic, and octic terms arise in ferroelectric and dielectric modeling. In particular, the sign pattern $\beta_4>0$, $\beta_6<0$, and $\beta_8>0$ appears in the Landau fit in \cite{Liu2025Nitride}; see also \cite{Li2005phenomeno,Chen2026BFO} for related eighth-order Landau expansions.
We divide \eqref{eq:landau_general} as
$f(u) = P_3(u)+\frac{\beta_4}{4}u^4+\frac{\beta_6}{6}u^6$ and $g(u) = \frac{\beta_8}{8}u^8$.
The function $\Phi=f+g$ is coercive. We take $q=3$, $\nu\in(0,1)$, and $p=3+\nu$. Since $D^3P_3$ is constant, $P_3$ does not affect $H_{3,\nu}^+(f)$. Moreover,
\[
f'''(x+h)-f'''(x) = h\left[6\beta_4 +20\beta_6\bigl(3x^2+3xh+h^2\bigr)\right],
\]
leading to
\[
H_{3,\nu}^+(f) = \supt{\substack{x\in\R\\h\neq0}} |h|^{1-\nu}\left(6\beta_4+20\beta_6\bigl(3x^2+3xh+h^2\bigr)\right)_+.
\]
Since $3x^2+3xh+h^2 = 3\left(x+\frac h2\right)^2+\frac14h^2$ and $\beta_6<0$, the supremum over $x$ is attained at $x=-h/2$, leading to
$H_{3,\nu}^+(f) = \sup_{s>0} s^{1-\nu}\bigl(6\beta_4+5\beta_6s^2\bigr)_+$.
The maximum is reached at $s_\nu^2 = \frac{6(1-\nu)\beta_4}{-5\beta_6(3-\nu)}$,
implying
\begin{equation}\label{eq:landau_fractional_modulus}
H_{3,\nu}^+(f) = \frac{12\beta_4}{3-\nu} \left(\frac{6(1-\nu)\beta_4}{-5\beta_6(3-\nu)}\right)^{\frac{1-\nu}{2}}<+\infty.
\end{equation}
Invoking Theorem~\ref{thm:fractional-characterization} yields
$f\in \maj{3}\left(3+\nu,\frac{\Gamma(\nu+1)}{\Gamma(3+\nu)} H_{3,\nu}^+(f)\right)$.
On the other hand,
$f'''(u)-f'''(0) = 6\beta_4u+20\beta_6u^3$. Since $\beta_6\neq0$, for every $\mu\in(0,1]$,
\[
\frac{|f'''(t)-f'''(0)|}{|t|^\mu} \to+\infty,\quad\text{as }t\to+\infty.
\]
Hence, $D^3f$ is not globally $\mu$-H\"older continuous for any $\mu\in(0,1]$, although we have $H_{3,\nu}^+(f)<+\infty$ for every $\nu\in(0,1)$.
\end{example}
At $p=q+1$, the characterization exhibits a parity-dependent structure.
\begin{theorem}[Parity structure at $p=q+1$]\label{thm:diagonal_parity}
Let $f\in C^{q+1}(\E)$.
\begin{enumerate}[label=(\textbf{\alph*}), font=\normalfont\bfseries, leftmargin=0.7cm]
\item\label{thm:diagonal_parity:even}
If $q$ is even, then $f\in\maj{q}(q+1,L)$ if and only if
\begin{equation}\label{eq:lip:thm:diagonal_parity}
    \|D^qf(y)-D^qf(x)\|\le q!L\|y-x\| \quad \forall x,y\in\E.
\end{equation}
In particular,
$L_{q+1,q}(f)=\frac1{q!}\operatorname{Lip}(D^qf)$,
where $\operatorname{Lip}(D^qf)$ denotes the optimal global Lipschitz constant of $D^qf$.

\item\label{thm:diagonal_parity:odd}
If $q$ is odd, then
\begin{equation}\label{eq:odd-best-constant}
L_{q+1,q}(f)=\frac1{q!}\supt{\substack{x\in\E\\\norm{h}=1}}\pos{D^{q+1}f(x)[h]^{q+1}}.
\end{equation}
In this case, membership in $\maj{q}(q+1,L)$ imposes no bound on the negative part of $D^{q+1}f(x)[h]^{q+1}$.
\end{enumerate}
\end{theorem}
\begin{proof}
Suppose first that $q$ is even and $f\in\maj{q}(q+1,L)$. By Theorem~\ref{thm:fractional-characterization}~\ref{thm:fractional-characterization:d},
\[
D^{q+1}f(x)[h]^{q+1} \leq q!L\norm{h}^{q+1} \quad \forall x,h\in\E.
\]
Since $q+1$ is odd, replacing $h$ by $-h$ implies
$-D^{q+1}f(x)[h]^{q+1} \leq q!L\norm{h}^{q+1}$. Hence,
\[
\abs{D^{q+1}f(x)[h]^{q+1}}\leq q!L\norm{h}^{q+1} \quad \forall x,h\in\E.
\]
Consequently, \eqref{eq:symmetric_tensor_norm} ensures $\|D^{q+1}f(x)\|\le q!L$, for all $x\in\E$.
For $x,y\in\E$, the fundamental theorem of calculus yields
\[
D^qf(y)-D^qf(x) = \int_0^1 D^{q+1}f(x+t(y-x))[y-x]\,dt,
\]
and therefore \eqref{eq:lip:thm:diagonal_parity} holds.
Conversely, suppose that $D^qf$ is globally Lipschitz with constant $q!L$. Fact~\ref{lem:holder_sufficient} with $\nu=1$ gives $f\in\maj{q}(q+1,L)$. Taking the smallest admissible constant in both directions yields
$L_{q+1,q}(f)=\frac1{q!}\operatorname{Lip}(D^qf)$.
Now, let us assume that $q$ is odd. Then,
$D^{q+1}f(x)[-h]^{q+1} = D^{q+1}f(x)[h]^{q+1}$.
Thus, replacing $h$ by $-h$ does not convert an upper bound into a lower bound. Formula \eqref{eq:odd-best-constant} follows directly from Theorem~\ref{thm:fractional-characterization}, while the negative
part of $D^{q+1}f(x)[h]^{q+1}$ remains unrestricted.
\end{proof}

The next example illustrates the odd-$q$ structure at the endpoint $p=q+1$. In particular, the tensor-majorant constant may be finite and positive even when $D^qf$ has no global H\"older modulus.

\begin{example}[Squared fourth-cumulant projection models]
Let $Z$ be a centered random vector in $\R^n$ with $\mathbb E[ZZ^\top]=I$ and $\mathbb E\|Z\|^4<+\infty$. Define $A(w):=\mathbb E[(w^\top Z)^4]$. For $\|w\|=1$, the fourth cumulant of $w^\top Z$ is $A(w)-3$.
Squared excess-kurtosis criteria arise naturally in projection pursuit and related statistical models; see, e.g., \cite{Loperfido2020KurtosisPP,Radojicic2021,Simon2001Squared}. Let us consider
\begin{equation}\label{eq:squared_kurtosis_general}
    \maxt{\|w\|=1}\bigl(A(w)-3\bigr)^2,
\end{equation}
or equivalently $f(w):=-\bigl(A(w)-3\bigr)^2$ and $g(w):=\delta_{\mathbb S^{n-1}}(w)$.
We take $q=3$ and $p=4$. For $x,h\in\R^n$, set $a=x^\top Z$, $b=h^\top Z$, and $m_{rs}:=\mathbb E[a^r b^s]$. We also have
\begin{equation}\label{eq:A2_fourth_derivative}
D^4(A^2)(x)[h]^4 = 48\bigl(m_{40}m_{04}+16m_{31}m_{13}+18m_{22}^2\bigr).
\end{equation}
The quantity in parentheses is nonnegative. Indeed,
\[
M:= \mathbb E \left[\begin{pmatrix}
a^2\\ ab\\ b^2
\end{pmatrix}
\begin{pmatrix}
a^2&ab&b^2
\end{pmatrix}\right] =
\begin{pmatrix}
m_{40}&m_{31}&m_{22}\\
m_{31}&m_{22}&m_{13}\\
m_{22}&m_{13}&m_{04}
\end{pmatrix}
\succeq 0.
\]
Since $f=-A^2+6A-9$ and $D^4A(x)[h]^4=24A(h)$, it can be concluded that
\[
D^4f(x)[h]^4 = 144A(h)-D^4(A^2)(x)[h]^4 \le 144A(h).
\]
At $x=0$ equality holds. Therefore, 
$L_{4,3}(f) = 24\maxt{\|h\|=1}\mathbb E[(h^\top Z)^4]$.
This constant is finite and positive. In fact, $\mathbb E[(h^\top Z)^4]\ge1$ for $\|h\|=1$.
On the other hand, for any unit vector $v$,
we have
$f(tv)=-A(v)^2t^8+6A(v)t^4-9$,
and hence
\[
D^3f(tv)[v]^3 = -336A(v)^2t^5+144A(v)t,
\]
i.e., for every $\mu\in(0,1]$,
\[
\frac{\|D^3f(tv)-D^3f(0)\|}{|t|^\mu} \ge \frac{\left|\bigl(D^3f(tv)-D^3f(0)\bigr)[v]^3\right|}{|t|^\mu} \to +\infty.
\]
As a result, $D^3f$ has no global H\"older modulus. In addition,
\[
D^4f(tv)[v]^4 = 144A(v)-1680A(v)^2t^4 \to -\infty, 
\]
showing explicitly that the negative part of the fourth-order directional derivative may be unbounded, as allowed by the odd-$q$ case of Theorem~\ref{thm:diagonal_parity}.
\end{example}
\paragraph{The case $p>q+1$}
When the regularization order exceeds the first omitted Taylor order $q+1$, the tensor-majorant class becomes considerably more restrictive. Indeed, whenever a finite majorization constant exists, the optimal one is necessarily zero.
\begin{theorem}[Remainder characterization: $p>q+1$]\label{thm:superdiagonal}
Let $f\in C^{q+1}(\E)$ and let $p>q+1$. Then the following statements are equivalent.
\begin{enumerate}[label=(\textbf{\alph*}), font=\normalfont\bfseries, leftmargin=0.7cm]
\item\label{thm:superdiagonal:a} $f\in\maj{q}(p,L)$ for some $L<+\infty$;
\item\label{thm:superdiagonal:b} $D^{q+1}f(x)[h]^{q+1}\le0$ for all $x,h\in\E$;
\item\label{thm:superdiagonal:c} $f(y)\le\Omega_{x,q}^f(y)$ for all $x,y\in\E$;
\item\label{thm:superdiagonal:d} $f\in\maj{q}(p,0)$. 
\end{enumerate}
Whenever these conditions hold, $L_{p,q}(f)=0$.
\end{theorem}
\begin{proof}
Suppose first that $f\in\maj{q}(p,L)$ for some $L<+\infty$. Fix $x,h\in\E$ and let $t>0$. Using $f\in C^{q+1}(\E)$ and \eqref{eq:maj_ineq}, we come to
\[
\frac{t^{q+1}}{(q+1)!}D^{q+1}f(x)[h]^{q+1}+o(t^{q+1})=f(x+th)-\Omega^f_{x,q}(x+th) \le \frac{L}{p}t^p\|h\|^p.
\]
Divide by $t^{q+1}$ and let $t\downarrow 0$. Since $p>q+1$, the right-hand side tends to zero. Thus, \ref{thm:superdiagonal:a} implies \ref{thm:superdiagonal:b}.
Suppose next that \ref{thm:superdiagonal:b} holds. Taylor's integral formula gives
\[
f(y)-\Omega_{x,q}^f(y) = \frac1{q!}\int_0^1(1-t)^q D^{q+1}f(x+t(y-x))[y-x]^{q+1}\,dt\le 0,
\]
i.e., $f(y)\le\Omega_{x,q}^f(y)$ for all $x,y\in\E$. Thus, \ref{thm:superdiagonal:b} implies \ref{thm:superdiagonal:c}. The equivalence of \ref{thm:superdiagonal:c} and \ref{thm:superdiagonal:d} follows immediately from Definition~\ref{def:maj}, while \ref{thm:superdiagonal:d} trivially implies
\ref{thm:superdiagonal:a}. Finally, whenever $L_{p,q}(f)<+\infty$, Assertion~\ref{thm:superdiagonal:a} and consequently \ref{thm:superdiagonal:d} hold, i.e., $L_{p,q}(f)=0$.
\end{proof}

Indeed, for $p>q+1$, a finite tensor majorization exists only if the $q$th-order Taylor model already majorizes $f$ globally; the optimal majorization constant is then zero, and the admissible function class does not depend on $p$.

\begin{corollary}[Parity structure for $p>q+1$]\label{cor:rigidity}
Let $f\in C^{q+1}(\E)$ and $p>q+1$.
\begin{enumerate}[label=(\textbf{\alph*}), font=\normalfont\bfseries, leftmargin=0.7cm]
\item If $q$ is even, then $f\in\maj{q}(p,L)$ for some finite $L$ if and only if $f$ is a polynomial of degree at most $q$.
\item If $q$ is odd, then $f\in\maj{q}(p,L)$ for some finite $L$ if and only if
\[
D^{q+1}f(x)[h]^{q+1}\le0 \quad \forall x,h\in\E.
\]
For $q=1$, this condition is equivalent to concavity of $f$.
\end{enumerate}
\end{corollary}
\begin{proof}
Let $q$ be even. Applying Theorem~\ref{thm:superdiagonal} to $h$ and $-h$ gives
$D^{q+1}f(x)[h]^{q+1}=0$. 
By \eqref{eq:symmetric_tensor_norm}, this implies $D^{q+1}f\equiv0$. Taylor's formula then shows that $f$ is a polynomial of degree at most $q$. The converse is immediate. The odd case is Theorem~\ref{thm:superdiagonal}, and the case $q=1$ is the usual second-order characterization of concavity.
\end{proof}

The next example shows that the case of odd $q$ in Corollary~\ref{cor:rigidity} arises naturally in even-power $L_r$-principal component analysis (PCA).

\begin{example}[Even-power $L_r$-PCA]
$L_p$-based PCA models have been studied for arbitrary positive exponents and extended to generalized and two-dimensional PCA settings; see, e.g.,
\cite{Kwak2014LpPCA,Liang2013LpGPCA,Wang2016G2DPCA}. We use the one-direction formulation \cite{Kwak2014LpPCA}
\[
\maxt{\|w\|=1}\frac1N\sum_{i=1}^N |w^\top x_i|^r,
\]
where $x_1,\ldots,x_N\in\R^n$ are the data vectors and $r>0$ is the PCA exponent.
Let $q\ge3$ be odd, and let $r\ge q+3$ be an even integer. Since $r$ is even, the corresponding composite minimization problem is $f_r(w) = -\frac1N\sum_{i=1}^N(w^\top x_i)^r$ and $g(w)=\delta_{\mathbb S^{n-1}}(w)$.
For every $w,h\in\R^n$,
\begin{equation}\label{eq:lp_pca_qplus1_derivative}
D^{q+1}f_r(w)[h]^{q+1} = -\frac{r!}{(r-q-1)!N} \sum_{i=1}^N (w^\top x_i)^{r-q-1} (h^\top x_i)^{q+1}.
\end{equation}
Since both $r-q-1$ and $q+1$ are even,
$D^{q+1}f_r(w)[h]^{q+1}\le 0$, for $w,h\in\R^n$.
Thus, Theorem~\ref{thm:fractional-characterization}~\ref{thm:fractional-characterization:d} gives $f_r\in\maj{q}(q+1,0)$, i.e., $L_{q+1,q}(f_r)=0$, while Theorem~\ref{thm:superdiagonal} yields $f_r\in\maj{q}(p,0)$, for all $p>q+1$. This conclusion does not follow from global H\"older continuity of $D^qf_r$. If the data are not all zero, choose a unit vector $v$ such that $\sum_{i=1}^N(v^\top x_i)^r>0$. Along $w=tv$,
\[
D^qf_r(tv)[v]^q = -\frac{r!}{(r-q)!N}t^{r-q}\sum_{i=1}^N(v^\top x_i)^r.
\]
Since $r-q\ge3$, for every $\mu\in(0,1]$,
\[
\frac{\|D^qf_r(tv)-D^qf_r(0)\|}{|t|^\mu} \ge \frac{\left|\bigl(D^qf_r(tv)-D^qf_r(0)\bigr)[v]^q \right|}{|t|^\mu} \to+\infty \quad\text{as }t\to+\infty,
\]
i.e., $D^qf_r$ is not globally $\mu$-H\"older continuous for any $\mu\in(0,1]$. In particular, taking $r=q+3$ shows that every odd $q\ge3$ admits an even-power PCA model satisfying $f_{q+3}\in\maj{q}(p,0)$ for all $ p\geq q+1$, while $D^qf_{q+3}$ does not have a global H\"older modulus. The choice $r=8$ and $q=5$ recovers the $L_8$-PCA instance.
\end{example}


\section{Tensor operator, envelope, and method}
\label{sec:Ten_env}
In this section, we study basic variational properties of TOP and TEE. In addition, we investigate the landscape and differentiability of TEE and the local single-valuedness of TOP. Finally, we use the tensor envelope as a Lyapunov function to analyze an exact tensor iteration and its iteration complexity for reaching stationarity.

\subsection{Basic properties and landscape}
We first establish comparison and descent properties of the tensor model. Let us define
\begin{equation}\label{eq:tensor_model_full}
\Model^{p,q}_\gamma(x,y):=\Omega_{x,q}^f(y)+g(y)+\frac1{p\gamma}\|y-x\|^p.
\end{equation}
The set of fixed-points is denoted by
$\Fix(\Tprox{\gf}{\gamma}{p,q}):=\{x\in\E:x\in\Tprox{\gf}{\gamma}{p,q}(x)\}$.
\begin{proposition}[Basic relations]\label{prop:tensor_basic_comparison}
Let $f\in\maj{q}(p,L)$. 
\begin{enumerate}[label=(\textbf{\alph*}), font=\normalfont\bfseries, leftmargin=0.7cm]
\item \label{prop:tensor_basic_comparison:a} For every $\gamma>0$ and all $x,y\in\E$,
\begin{equation}\label{eq:model_majorizes_phi_quantitatively}
\Model^{p,q}_\gamma(x,y) \ge \gf(y)+\frac{\gamma^{-1}-L}{p}\|y-x\|^p.
\end{equation}

\item \label{prop:tensor_basic_comparison:b} If $0<\gamma\le L^{-1}$, then, for every $x\in\E$, 
$\inf_{\E}\gf \le \fgam{\gf}{p,q}{\gamma}(x)\le \gf(x)$
and 
\[
\inft{x\in\E}\fgam{\gf}{p,q}{\gamma}(x) = \inft{x\in\E}\gf(x).
\]

\item \label{prop:tensor_basic_comparison:c}
If $0<\gamma\le L^{-1}$ and $y\in\Tprox{\gf}{\gamma}{p,q}(x)$, then
\begin{equation}\label{eq:tensor_descent_quantitative}
\gf(y)+\frac{\gamma^{-1}-L}{p}\|y-x\|^p \le\fgam{\gf}{p,q}{\gamma}(x)\le\gf(x).
\end{equation}

\item \label{prop:tensor_basic_comparison:d} If $0<\gamma<\mu$, then 
$\fgam{\gf}{p,q}{\mu}(x)\le\fgam{\gf}{p,q}{\gamma}(x)$.
\end{enumerate}
\end{proposition}
\begin{proof}
Since $f\in\maj{q}(p,L)$, we have
$\Omega_{x,q}^f(y) \ge f(y)-\frac{L}{p}\|y-x\|^p$.
Adding $g(y)$ and $(p\gamma)^{-1}\|y-x\|^p$ gives \eqref{eq:model_majorizes_phi_quantitatively}. Taking the infimum in $y$ yields $\inf_{\E}\gf \le \fgam{\gf}{p,q}{\gamma}(x)$ and using $y=x$ in \eqref{eq:value} gives $\fgam{\gf}{p,q}{\gamma}(x)\le \gf(x)$. 
These inequalities also give equality of the infima.
If $y\in\Tprox{\gf}{\gamma}{p,q}(x)$, then $\fgam{\gf}{p,q}{\gamma}(x)=\Model^{p,q}_\gamma(x,y)$,
so \eqref{eq:tensor_descent_quantitative} follows from \eqref{eq:model_majorizes_phi_quantitatively} and
Assertion~\ref{prop:tensor_basic_comparison:b}. Finally, the monotonicity in $\gamma$ follows directly from the definition.
\end{proof}

Next, we establish the basic existence and continuity properties of TOP and TEE.

\begin{theorem}[Basic properties of TOP and TEE]\label{thm:objective_prox_wellposed}
Let $f\in\maj{q}(p,L)$ and suppose that $\gf$ is high-order prox-bounded with threshold
$\tau:=\gamma^{\gf,p}$. Set $\ov\gamma_{L,\tau} := \frac{1}{L+\tau^{-1}}$, let $0<\gamma<\ov\gamma_{L,\tau}$, and define
$\delta_\gamma := \frac{\gamma}{1-\gamma L}$.
Then $\delta_\gamma<\tau$ and
\begin{equation}\label{eq:HOME_lower_tensor}
\fgam{\gf}{p}{\delta_\gamma}(x) \le \fgam{\gf}{p,q}{\gamma}(x) \le \gf(x)
\quad \forall x\in\E.
\end{equation}
In addition, the following properties hold.
\begin{enumerate}[label=(\textbf{\alph*}), font=\normalfont\bfseries, leftmargin=0.7cm]
\item\label{thm:objective_prox_wellposed:a} For every $x\in\E$, the function $y\mapsto\Model^{p,q}_\gamma(x,y)$ is proper, lsc, and level-bounded;
\item The set $\Tprox{\gf}{\gamma}{p,q}(x)$ is nonempty and compact for every $x\in\E$;
\item The tensor envelope $\fgam{\gf}{p,q}{\gamma}$ is finite-valued on $\E$;
\item $\Model^{p,q}_\gamma(x,y)$ is level-bounded in $y$ locally uniformly in $(x,\gamma)$ on
$\E\times(0,\ov\gamma_{L,\tau})$;

\item The map $(x,\gamma)\mapsto\fgam{\gf}{p,q}{\gamma}(x)$ is continuous on
$\E\times(0,\ov\gamma_{L,\tau})$;

\item\label{thm:objective_prox_wellposed:f} If $x^k\to x$, $\gamma_k\to\gamma\in(0,\ov\gamma_{L,\tau})$, and
$y^k\in\Tprox{\gf}{\gamma_k}{p,q}(x^k)$,
then $\{y^k\}_{k\in\Nz}$ is bounded and each cluster point belongs to $\Tprox{\gf}{\gamma}{p,q}(x)$;

\item \label{thm:objective_prox_wellposed:g} For each $x\in\E$, $\fgam{\gf}{p,q}{\gamma}(x)\uparrow\gf(x)$ as $\gamma\downarrow 0$.
\end{enumerate}
\end{theorem}
\begin{proof}
Since $\gamma<\ov\gamma_{L,\tau}$, $\delta_\gamma=\frac{1}{\gamma^{-1}-L}<\tau$.
By \eqref{eq:model_majorizes_phi_quantitatively},
\[
\Model^{p,q}_\gamma(x,y)\ge \gf(y)+\frac{\gamma^{-1}-L}{p}\norm{y-x}^p= \gf(y)+\frac1{p\delta_\gamma}\|y-x\|^p.
\]
Taking the infimum in $y$ gives the first inequality in \eqref{eq:HOME_lower_tensor}. The second is obtained from Proposition~\ref{prop:tensor_basic_comparison}~\ref{prop:tensor_basic_comparison:b}.
Since $\delta_\gamma<\tau$, Fact~\ref{fact:level-bound+locally uniform} implies that 
$y\mapsto \gf(y)+\frac{1}{p\delta_\gamma}\|y-x\|^p$ is level-bounded. Hence 
$y\mapsto\Model^{p,q}_\gamma(x,y)$ is also level-bounded. It is proper and lsc because $g$ is proper and lsc and the remaining terms are continuous. Since it is also level-bounded, its minimum is attained, which proves that TOP is nonempty and compact and that TEE is finite-valued.
We next prove the locally uniform statement. Fix $(\ov x,\ov\gamma)\in\E\times(0,\ov\gamma_{L,\tau})$ and choose $\gamma_2$ so that
$\ov\gamma<\gamma_2<\ov\gamma_{L,\tau}$.
After restricting $\gamma$ to a neighborhood of $\ov\gamma$, 
we may assume $\gamma\leq\gamma_2$. Since $\delta_\gamma\leq\delta_{\gamma_2}<\tau$,
\[
\Model^{p,q}_\gamma(x,y) \ge \gf(y)+\frac1{p\delta_{\gamma_2}}\|y-x\|^p.
\]
By Fact~\ref{fact:level-bound+locally uniform}, the right-hand side is level-bounded in $y$ locally uniformly in $x$. Hence
$\Model_\gamma^{p,q}(x,y)$ is level-bounded in $y$ locally uniformly in $(x,\gamma)$.
The continuity of the value function and the boundedness and outer-semicontinuity of the corresponding argmin mapping now follow from the standard parametric-minimization theorem \cite[Theorem~1.17]{Rockafellar09}.
Finally, let $\gamma_k\downarrow0$ and fix $\delta\in(0,\tau)$. If $x\in\dom\gf$, choose $y^k\in\Tprox{\gf}{\gamma_k}{p,q}(x)$. Since
$\gf(y^k)\ge \fgam{\gf}{p}{\delta}(x)-\frac{1}{p\delta}\|y^k-x\|^p$, \eqref{eq:tensor_descent_quantitative} gives
\[
\fgam{\gf}{p}{\delta}(x)+\frac{\gamma_k^{-1}-L-\delta^{-1}}{p}\|y^k-x\|^p\le \gf(x).
\]
Hence $y^k\to x$. Lower semicontinuity gives $\liminf_k\fgam{\gf}{p,q}{\gamma_k}(x)\ge\gf(x)$, while the reverse inequality follows from $\fgam{\gf}{p,q}{\gamma_k}(x)\le\gf(x)$.
If $x\notin\dom\gf$ and $\fgam{\gf}{p,q}{\gamma_k}(x)$ were bounded above along a subsequence, the same lower estimate would force the corresponding minimizers to converge to $x$ while their objective values remain bounded above, contradicting lower semicontinuity at $\gf(x)=+\infty$. Thus $\fgam{\gf}{p,q}{\gamma}(x)\to\gf(x)$ for every $x\in\E$. Monotonicity in $\gamma$ then gives the asserted upward convergence.
\end{proof}

An immediate simplification occurs when the original objective is bounded below.
\begin{corollary}[Bounded-below objectives]\label{cor:bounded_below_wellposed}
Let $f\in\maj{q}(p,L)$ and let $\gf$ be bounded below. Then, for every $0<\gamma<L^{-1}$, all conclusions of Theorem~\ref{thm:objective_prox_wellposed} hold. 
\end{corollary}
\begin{proof}
Since $\gf$ is bounded below, $\gamma^{\gf,p}=+\infty$, resulting in $\ov\gamma_{L,+\infty}=L^{-1}$, and the result follows from Theorem~\ref{thm:objective_prox_wellposed}.
\end{proof}

The subsequent example shows that the parameter range in Theorem~\ref{thm:objective_prox_wellposed} is sharp.
\begin{example}[Sharpness of the parameter range]\label{ex:sharp_gamma_threshold}
Let $q\ge 1$ be odd, set $p=q+1$, and consider 
$f(t)=\frac{L}{q+1}t^{q+1}$ and $g(t)=-\frac{L+\tau^{-1}}{q+1}t^{q+1}$ on $\R$, 
with $L>0$ and $0<\tau<+\infty$. Then
$\gf(t)=-\frac{1}{(q+1)\tau}t^{q+1}$,
and the prox-boundedness threshold of $\gf$ is $\tau$. Indeed, for a HOME parameter $\delta$, the leading coefficient of $\gf(y)+\frac{1}{p\delta}|y-x|^p$ is $\frac1p(\delta^{-1}-\tau^{-1})$, which is
 positive for $\delta<\tau$ and negative for $\delta>\tau$. Moreover,
$f^{(q+1)}(t)=q!L$,
and hence Theorem~\ref{thm:fractional-characterization}\ref{thm:fractional-characterization:d} gives $f\in\maj{q}(q+1,L)$. Since $\Omega_{0,q}^f(y)=0$, the tensor model at $x=0$ is $\Model^{q+1,q}_\gamma(0,y) = \frac{\gamma^{-1}-L-\tau^{-1}}{q+1}|y|^{q+1}$, i.e., it is level-bounded when $\gamma<\ov\gamma_{L,\tau}$, it is identically zero at $\gamma=\ov\gamma_{L,\tau}$ and hence has the noncompact argmin set $\R$, and it is unbounded below whenever $\gamma>\ov\gamma_{L,\tau}$. Hence, the parameter range in Theorem~\ref{thm:objective_prox_wellposed} cannot be enlarged, in general.
\end{example}

Another well-definedness route, which does not require a tensor-majorant assumption, follows by controlling the Taylor polynomial with the power regularizer.

\begin{proposition}[Alternative well-definedness route]\label{prop:g_prox_wellposed}
Suppose that $p>q$ and that $g$ is high-order prox-bounded with threshold $\gamma^{g,p}>0$. Then, for every $\gamma\in(0,\gamma^{g,p})$, the conclusions in Theorem~\ref{thm:objective_prox_wellposed}~\ref{thm:objective_prox_wellposed:a}--\ref{thm:objective_prox_wellposed:g} hold with $(0,\gamma^{g,p})$.
\end{proposition}
\begin{proof}
Fix $(\ov x,\ov\gamma)\in\E\times(0,\gamma^{g,p})$ and choose $\eta$ with $\ov\gamma<\eta<\gamma^{g,p}$. For $(x,\gamma)$ near $(\ov x,\ov\gamma)$, we may assume $\gamma<\eta$. Decompose
\[
\Model^{p,q}_\gamma(x,y)=g(y)+\frac1{p\eta}\|y-x\|^p+\Omega^f_{x,q}(y)+\frac{\gamma^{-1}-\eta^{-1}}p\|y-x\|^p.
\]
The first two terms are level-bounded in $y$, locally uniformly in $x$, by high-order prox-boundedness of $g$. Since $p>q$, the positive $p$th-order term in the remaining part dominates the degree-$q$ Taylor polynomial uniformly for $x$ in a compact neighborhood of $\ov x$. Hence, the tensor models are level-bounded locally uniformly in $(x,\gamma)$. Properness and lower semicontinuity then give nonempty compact argmin sets and finite values, while the standard parametric-minimization theorem yields continuity of TEE as well as local boundedness and outer
semicontinuity of TOP. For the recovery property, let $\gamma_k\downarrow0$ and choose tensor
minimizers $y^k$. If $x\in\dom g$, comparison with $y=x$ and the above decomposition imply $y^k\to x$, and lower semicontinuity gives $\fgam{\gf}{p,q}{\gamma_k}(x)\to\gf(x)$. If $x\notin\dom g$ and the envelope values were bounded above along a subsequence, the same coercive decomposition would again force $y^k\to x$ while keeping $g(y^k)$ bounded above, contradicting lower semicontinuity of $g$ at $x$, where $g(x)=+\infty$. Monotonicity in $\gamma$ follows directly from the definition.
\end{proof}

Next, we relate the minimizers and critical points of $\gf$ to the fixed points of TOP and the minimizers of TEE.
\begin{proposition}[Minimizers, fixed points, and critical points]\label{prop:fixed_critical}
The following statements hold:
\begin{enumerate}[label=(\textbf{\alph*}), font=\normalfont\bfseries, leftmargin=0.7cm]
\item \label{prop:fixed_critical:a} For every $p>1$, $q\ge1$, and $\gamma>0$,
$\Fix(\Tprox{\gf}{\gamma}{p,q}) \subseteq \crit(\gf)$.

\item \label{prop:fixed_critical:b} Let $f\in\maj{q}(p,L)$ and $\gamma\in (0,L^{-1})$. Then,
for every $x^\star\in\argmint{x\in\E}\gf(x)$,
\[
    \Tprox{\gf}{\gamma}{p,q}(x^\star)=\{x^\star\}, \quad\argmint{x\in\E}\gf(x) \subseteq \Fix(\Tprox{\gf}{\gamma}{p,q}) \subseteq \crit(\gf).
\]
If, in addition, $\fgam{\gf}{p,q}{\gamma}$ is finite-valued on $\E$, then
\begin{equation}\label{eq:equality_minimizers_new}
\argmint{x\in\E}\fgam{\gf}{p,q}{\gamma}(x) = \argmint{x\in\E}\gf(x).
\end{equation}
\end{enumerate}
\end{proposition}
\begin{proof}
\ref{prop:fixed_critical:a} Let $x\in\Fix(\Tprox{\gf}{\gamma}{p,q})$.
Fermat's rule and the smooth sum rule ensure
\[
0\in \nabla_y\Omega_{x,q}^f(x) + \nabla_y\left( \frac{1}{p\gamma}\|y-x\|^p\right)_{y=x} + \widehat\partial g(x)
=\nabla_y\Omega_{x,q}^f(x) + \widehat\partial g(x).
\]
Since $\nabla_y\Omega_{x,q}^f(x)=\nabla f(x)$, we obtain
$0\in\nabla f(x)+\widehat\partial g(x) \subseteq \nabla f(x)+\partial g(x)$,
and consequently $x\in\crit(\gf)$.
\\
\ref{prop:fixed_critical:b}
For $x^\star\in\argmin{x\in\E}\gf(x)$, \eqref{eq:model_majorizes_phi_quantitatively} yields $\Model^{p,q}_\gamma(x^\star,y)\ge\gf(y)>\gf(x^\star)=\Model^{p,q}_\gamma(x^\star,x^\star)$ 
for every $y\ne x^\star$,
that is
$\Tprox{\gf}{\gamma}{p,q}(x^\star)=\{x^\star\}$.
Invoking Assertion~\ref{prop:fixed_critical:a} leads to the second inclusion.
Assume now that TEE is finite-valued and let $\ov x$ minimize it. By Proposition~\ref{prop:tensor_basic_comparison}, we have $\fgam{\gf}{p,q}{\gamma}(\ov x)=\inf\gf> -\infty$. Let us choose $y^j$ such that $\Model^{p,q}_\gamma(\ov x,y^j)\le\inf\gf+j^{-1}$. Then \eqref{eq:model_majorizes_phi_quantitatively} implies $\frac{\gamma^{-1}-L}{p}\|y^j-\ov x\|^p\le j^{-1}$, i.e., $y^j\to\ov x$, while $\inf\gf\le\gf(y^j)\le\Model^{p,q}_\gamma(\ov x,y^j)\to\inf\gf$. Lower semicontinuity gives $\gf(\ov x)=\inf\gf$. The reverse inclusion follows from the singleton property above and Proposition~\ref{prop:tensor_basic_comparison}.
\end{proof}

\begin{proposition}[Level-boundedness and coercivity of TEE]\label{prop:level_boundedness_tee}
Let $f\in\maj{q}(p,L)$ and suppose that $\gf$ is high-order prox-bounded with threshold $\tau>0$. If $0<\gamma<\ov\gamma_{L,\tau}$, where $\ov\gamma_{L,\tau} := \frac{1}{L+\tau^{-1}}$, then
$\fgam{\gf}{p,q}{\gamma}$  is level-bounded if and only if $\gf$  is level-bounded.
Equivalently, $\fgam{\gf}{p,q}{\gamma}$ is coercive if and only if $\gf$ is coercive.
\end{proposition}
\begin{proof}
Since $\fgam{\gf}{p,q}{\gamma}(x)\le\gf(x)$, every sublevel set of $\gf$ is contained in the corresponding sublevel set of TEE. Hence level-boundedness of TEE implies level-boundedness of $\gf$. Conversely, suppose that $\gf$ is level-bounded. Since $\gf$ is proper and lsc, it attains a finite minimum $\gf^*:=\min_{\E}\gf$. Fix $a\ge \gf^*$ and let $\fgam{\gf}{p,q}{\gamma}(x)\le a$. Choosing $y\in\Tprox{\gf}{\gamma}{p,q}(x)$, \eqref{eq:tensor_descent_quantitative} yields
\[
a\ge\fgam{\gf}{p,q}{\gamma}(x) \ge\gf(y)+\frac{\gamma^{-1}-L}{p}\|y-x\|^p,
\]
i.e., $\gf(y)\le a$ and
\[
\|x-y\|^p\le\frac{p(a-\gf^*)}{\gamma^{-1}-L}.
\]
Consequently, $y$ belongs to the bounded sublevel set $\{\gf\leq a\}$ and $x$ remains at a uniformly bounded distance from $y$. Therefore, the sublevel set $\{\fgam{\gf}{p,q}{\gamma}\leq a\}$ is bounded. If $a<\gf^*$, it is empty. Since coercivity is equivalent to level-boundedness, the corresponding coercivity
equivalence follows immediately.
\end{proof}

\begin{remark}
If $\gf$ is level-bounded, then every exact tensor trajectory $x^{k+1}\in\Tprox{\gf}{\gamma}{p,q}(x^k)$ is bounded. Indeed,
$\gf(x^{k+1})\leq \fgam{\gf}{p,q}{\gamma}(x^k) \leq \gf(x^k)$,
i.e., $\{\gf(x^k)\}_{k\in\Nz}$ is nonincreasing, and all iterates remain in the
bounded sublevel set $\{x:\gf(x)\leq\gf(x^0)\}$.
\end{remark}

\subsection{Single-valuedness of TOP and differentiability of TEE}
In this section, we study the differentiability of TEE. We first characterize its Fr\'echet differentiability in terms of the gradients associated with solutions of the tensor subproblem, showing in particular that single-valuedness of TOP implies differentiability of TEE. We then
provide sufficient conditions, based on $p$-calmness and prox-regularity, for local single-valuedness of TOP and hence local smoothness of TEE.

We first reduce the tensor subproblem locally to minimization over a fixed compact set, which will be used in the differentiability analysis of TEE.
\begin{lemma}[Local compact reduction]\label{lem:local_compact_reduction}
Fix $\ov x\in\E$ and $\gamma>0$. Suppose that $\Tprox{\gf}{\gamma}{p,q}(x)$ is nonempty for $x$ near $\ov x$ and that $\Model^{p,q}_\gamma(x,\cdot)$ is level-bounded locally uniformly in $x$ around $\ov x$. Then there exist a neighborhood $U$ of $\ov x$ and a compact set $K\subset\E$ such that
$\Tprox{\gf}{\gamma}{p,q}(x)\subseteq K$, for each $x\in U$ and
\begin{equation}\label{eq:local_compact_reduced_value}
\fgam{\gf}{p,q}{\gamma}(x) = \mint{y\in K}\Model^{p,q}_\gamma(x,y) \quad \forall x\in U.
\end{equation}
\end{lemma}
\begin{proof}
Choose $y_0\in\dom g$. Since $x\mapsto\Model^{p,q}_\gamma(x,y_0)$ is continuous, there exists a neighborhood $U_1$ of $\ov x$ and $\alpha\in\R$ such that
$\Model^{p,q}_\gamma(x,y_0)\le\alpha$, for each $x\in U_1$.
Hence, for every
$y\in\Tprox{\gf}{\gamma}{p,q}(x)$, we have
$\Model^{p,q}_\gamma(x,y)\le \Model^{p,q}_\gamma(x,y_0) \le\alpha$.
By local uniform level-boundedness, after possibly shrinking the neighborhood to some $U\subseteq U_1$, the set
\[
B:= \bigcup_{x\in U}\left\{y\in\E: \Model^{p,q}_\gamma(x,y)\le\alpha \right\},
\]
is bounded. Hence, $K:=\bs{\rm cl}(B)$ is compact and 
$\Tprox{\gf}{\gamma}{p,q}(x)\subseteq K$, for each $x\in U$.
\end{proof}

Now, let us assume that $f\in C^{q+1}(\E)$ and write
\begin{equation}\label{eq:S_smooth_part}
S_{q,p,\gamma}(x,y) := \Omega_{x,q}^f(y)+\frac1{p\gamma}\|y-x\|^p, 
\end{equation}
i.e., $\Model^{p,q}_\gamma(x,y)=S_{q,p,\gamma}(x,y)+g(y)$.
A direct differentiation with respect to the base point implies
\begin{equation}\label{eq:base_gradient}
\nabla_xS_{q,p,\gamma}(x,y) = \frac1{q!}D^{q+1}f(x)[\,\cdot\,,y-x,\ldots,y-x] - \frac1\gamma\|y-x\|^{p-2}B(y-x).
\end{equation}
The terms arising from differentiating the Taylor coefficients and increments cancel successively, explaining why the differentiability analysis requires $f\in C^{q+1}(\E)$.

The next lemma is the standard Danskin formula specialized in the local compact reduction in Lemma~\ref{lem:local_compact_reduction}. We include a short proof because $g$ may be extended-valued and only lsc; see \cite[Chapter~2]{Clarke90} and \cite[Chapter~10]{Rockafellar09}.
\begin{lemma}[Local Danskin formula]\label{fact:danskin_tensor}
Let $f\in C^{q+1}(\E)$ and fix $\ov x\in\E$. Suppose that the assumptions of Lemma~\ref{lem:local_compact_reduction} hold around $\ov x$. Then $\fgam{\gf}{p,q}{\gamma}$ is locally Lipschitz around $\ov x$ and is Hadamard directionally differentiable at $\ov x$, with
\begin{equation}\label{eq:danskin_directional}
(\fgam{\gf}{p,q}{\gamma})'(\ov x;d) = \mint{y\in\Tprox{\gf}{\gamma}{p,q}(\ov x)}
\left\langle\nabla_xS_{q,p,\gamma}(\ov x,y),d\right\rangle \quad \forall d\in\E.
\end{equation}
\end{lemma}
\begin{proof}
Let $U$ and $K$ be given by Lemma~\ref{lem:local_compact_reduction}. After shrinking $U$ if necessary, we may assume that $\bs{\rm cl}(U)$ is compact. Since $(x,y)\mapsto\nabla_xS_{q,p,\gamma}(x,y)$ is
continuous, there exists $C>0$ such that $\|\nabla_xS_{q,p,\gamma}(x,y)\|_*\le C$, 
for $(x,y)\in\bs{\rm cl}(U)\times K$.
Hence, uniformly for $y\in K$,
\[
|S_{q,p,\gamma}(x,y)-S_{q,p,\gamma}(z,y)| \le C\|x-z\|,
\]
for $x,z$ sufficiently close to $\ov x$. Using \eqref{eq:local_compact_reduced_value} gives the local Lipschitz continuity of TEE.
Fix $d\in\E$. To establish Hadamard directional differentiability, let $t_k\downarrow 0$ and $d^k\to d$, and set $x^k:=\ov x+t_kd^k$. For any $y\in\Tprox{\gf}{\gamma}{p,q}(\ov x)$, 
\[
\frac{\fgam{\gf}{p,q}{\gamma}(x^k)-\fgam{\gf}{p,q}{\gamma}(\ov x)}{t_k}
\le \frac{S_{q,p,\gamma}(x^k,y)-S_{q,p,\gamma}(\ov x,y)}{t_k}.
\]
Taking the upper limit and then minimizing over $\Tprox{\gf}{\gamma}{p,q}(\ov x)$ lead to 
\[
\limsupt{k\to\infty}\frac{\fgam{\gf}{p,q}{\gamma}(x^k)-\fgam{\gf}{p,q}{\gamma}(\ov x)}{t_k}
\leq \mint{y\in\Tprox{\gf}{\gamma}{p,q}(\ov x)}\ip{\nabla_xS_{q,p,\gamma}(\ov x,y)}{d}.
\]
For the reverse inequality, choose $y^k\in\Tprox{\gf}{\gamma}{p,q}(x^k)$.
Since $y^k\in K$, after passing to a subsequence and taking the lower limit, we may assume $y^k\to\ov y\in K$. Moreover,
\[
\Model_\gamma^{p,q}(\ov x,y^k) = \fgam{\gf}{p,q}{\gamma}(x^k)-\big(S_{q,p,\gamma}(x^k,y^k)-S_{q,p,\gamma}(\ov x,y^k)\big).
\]
The first term tends to $\fgam{\gf}{p,q}{\gamma}(\ov x)$ by local Lipschitz continuity, while
the parenthesized term is $O(t_k)$, i.e., $\Model^{p,q}_\gamma(\ov x,y^k)\to\fgam{\gf}{p,q}{\gamma}(\ov x)$.
Lower semicontinuity yields $\ov y\in\Tprox{\gf}{\gamma}{p,q}(\ov x)$. Using the optimality of $y^k$ at $x^k$ and the definition of the envelope at $\ov x$, we obtain
\[
\frac{\fgam{\gf}{p,q}{\gamma}(x^k)-\fgam{\gf}{p,q}{\gamma}(\ov x)}{t_k}
\ge \frac{S_{q,p,\gamma}(x^k,y^k)-S_{q,p,\gamma}(\ov x,y^k)}{t_k}.
\]
The right-hand side converges to
$\langle\nabla_xS_{q,p,\gamma}(\ov x,\ov y),d\rangle$,
which is at least the minimum in \eqref{eq:danskin_directional}. As such, the upper and lower limits coincide, proving the claimed Hadamard directional derivative.
\end{proof}

For $x\in\E$, let us define
\begin{equation}\label{eq:active_gradient_set}
\mathcal A_\gamma^{p,q}(x) := \left\{\nabla_xS_{q,p,\gamma}(x,y) : y\in\Tprox{\gf}{\gamma}{p,q}(x)
\right\}.
\end{equation}

\begin{theorem}[Differentiability criterion]\label{thm:active_gradient_differentiability}
Let the assumptions of Lemma~\ref{fact:danskin_tensor} hold. Then $\fgam{\gf}{p,q}{\gamma}$ is Fr\'echet differentiable at $\ov x$ if and only if $\mathcal A_\gamma^{p,q}(\ov x)$ is a singleton. If
$\mathcal A_\gamma^{p,q}(\ov x)=\{\zeta\}$ then $\nabla\fgam{\gf}{p,q}{\gamma}(\ov x)=\zeta$, and consequently, for every $y\in\Tprox{\gf}{\gamma}{p,q}(\ov x)$, we have
$\nabla\fgam{\gf}{p,q}{\gamma}(\ov x) = \nabla_xS_{q,p,\gamma}(\ov x,y)$, i.e.,
\begin{equation}\label{eq:active_gradient_identity}
\nabla\fgam{\gf}{p,q}{\gamma}(\ov x) = \frac1{q!}D^{q+1}f(\ov x)[\,\cdot\,,y-\ov x,\ldots,y-\ov x]
- \frac1\gamma\|y-\ov x\|^{p-2}B(y-\ov x).
\end{equation}
\end{theorem}
\begin{proof}
If $\fgam{\gf}{p,q}{\gamma}$ is Fr\'echet differentiable at $\ov x$, then Lemma~\ref{fact:danskin_tensor} implies
\[
\langle\nabla\fgam{\gf}{p,q}{\gamma}(\ov x),d\rangle
= \mint{\zeta\in\mathcal A_\gamma^{p,q}(\ov x)}\langle \zeta,d\rangle
\quad \forall d\in\E.
\]
Fixing $\zeta\in\mathcal A_\gamma^{p,q}(x)$, we come to
$\langle\nabla\fgam{\gf}{p,q}{\gamma}(\ov x),d\rangle\leq\langle \zeta,d\rangle$. 
Applying the same inequality to $-d$ gives the reverse inequality, i.e., 
\[
\langle \zeta-\nabla\fgam{\gf}{p,q}{\gamma}(\ov x),d\rangle=0 \quad \forall d\in\E.
\]
As a consequence, every element of $\mathcal A^{p,q}_\gamma(\ov x)$ equals $\nabla\fgam{\gf}{p,q}{\gamma}(\ov x)$, and consequently $\mathcal A_\gamma^{p,q}(x)$ is a singleton.
Conversely, suppose that $\mathcal A_\gamma^{p,q}(\ov x)=\{\zeta\}$. Now, Lemma~\ref{fact:danskin_tensor} implies
\[
(\fgam{\gf}{p,q}{\gamma})'(x;d) = \ip{\zeta}{d} \quad \forall d\in\E,
\]
which implies that the Hadamard directional derivative is linear. Since TEE is locally Lipschitz and $\E$ is finite-dimensional, this implies Fr\'echet differentiability at $x$, with $\nabla\fgam{\gf}{p,q}{\gamma}(\ov x)=\zeta$. 
\end{proof}

Theorem~\ref{thm:active_gradient_differentiability} shows that the differentiability of TEE is characterized by the equality of the base-point gradients associated with all tensor solutions. In particular, single-valuedness of TOP is sufficient but need not be assumed a priori.

\begin{corollary}[Single-valued TOP and differentiability of TEE]
\label{cor:single_valued_differentiability}
Suppose that $f\in C^{q+1}(\E)$ and that the local compact-reduction assumptions of Lemma~\ref{lem:local_compact_reduction} hold at every point of an open set $U\subset\E$. If $\Tprox{\gf}{\gamma}{p,q}$ is single-valued on $U$, then $\Tprox{\gf}{\gamma}{p,q}$ is continuous on $U$ and
$\fgam{\gf}{p,q}{\gamma}\in C^1(U)$.
\end{corollary}
\begin{proof}
Fix $x\in U$ and let $x^k\to x$. By Lemma~\ref{lem:local_compact_reduction}, the points $y^k:=\Tprox{\gf}{\gamma}{p,q}(x^k)$ lie, for large $k$, in a common compact set. Every cluster point $\ov y$ satisfies $\ov y\in\Tprox{\gf}{\gamma}{p,q}(x)$ by the same lower-semicontinuity argument used in Lemma~\ref{fact:danskin_tensor}. Single-valuedness therefore forces $y^k\to\Tprox{\gf}{\gamma}{p,q}(x)$, so TOP is continuous on $U$. Theorem~\ref{thm:active_gradient_differentiability} gives the differentiability of TEE, and \eqref{eq:active_gradient_identity}, together with the continuity of TOP and $\nabla_xS_{q,p,\gamma}$, yields continuity of the gradient.
\end{proof}

We next give conditions that ensure the local single-valuedness required in Corollary~\ref{cor:single_valued_differentiability}.
Following \cite{Kabganidiff}, a point $\ov x\in\dom\gf$ is called $p$-calm for $\gf$ with constant $M>0$ if
\begin{equation}\label{eq:p_calm_tee}
\gf(y)+M\|y-\ov x\|^p>\gf(\ov x) \quad \forall\,y\neq\ov x.
\end{equation}
For $r\ge 2$, the function $\gf$ is said to be $r$-prox-regular at $\ov x$ for $\ov v\in\partial\gf(\ov x)$ if there exist $\varepsilon>0$ and $\rho\ge 0$ such that
\begin{equation}\label{eq:r_prox_regular_tee}
    \gf(z) \ge \gf(y)+\langle v,z-y\rangle -\frac{\rho}{2}\|z-y\|^r,
\end{equation}
for $y,z\in \mb(\ov x;\varepsilon)$ whenever $v\in\partial\gf(y)$, $\|v-\ov v\|_*<\varepsilon$, and $\gf(y)<\gf(\ov x)+\varepsilon$.
For $r=2$, this reduces to the usual prox-regularity condition \cite{Rockafellar09}.
Notice that every $p$-calm point $\ov x$ satisfies $0\in\widehat\partial\gf(\ov x)\subseteq\partial\gf(\ov x)$.

\begin{theorem}[Differentiability around $p$-calm points]\label{thm:calm_prox_regular_tee}
Let $f\in C^{q+1}(\E)\cap\mathcal M^q(p,L)$ and let $\ov x$ be a $p$-calm point of $\gf$ with constant $M>0$.
\begin{enumerate}[label=(\textbf{\alph*}), font=\normalfont\bfseries, leftmargin=0.7cm]
\item \label{thm:calm_prox_regular_tee:a} For every $0<\gamma<\frac{1}{L+pM}$, one has
$\Tprox{\gf}{\gamma}{p,q}(\ov x)=\{\ov x\}$ and $\mathcal A_\gamma^{p,q}(\ov x)=\{0\}$. Consequently, $\fgam{\gf}{p,q}{\gamma}$ is Fr\'echet differentiable
at $\ov x$ and $\nabla\fgam{\gf}{p,q}{\gamma}(\ov x)=0$. 

\item \label{thm:calm_prox_regular_tee:b} Suppose, in addition, that $1<p\le q+1$ and that $\gf$ is $r$-prox-regular at $\ov x$ for $0\in\partial\gf(\ov x)$ for some
$r\ge\max\{2,p\}$. Then there exists $\ov\gamma>0$ such that, for every $\gamma\in(0,\ov\gamma)$, there is a neighborhood $U_\gamma$ of $\ov x$ on which $\Tprox{\gf}{\gamma}{p,q}$ is single-valued and continuous. Consequently, $\fgam{\gf}{p,q}{\gamma}\in C^1(U_\gamma)$.
\end{enumerate}
\end{theorem}
\begin{proof}
\ref{thm:calm_prox_regular_tee:a} The $p$-calmness condition implies that $\gf$ is high-order prox-bounded of order $p$ \cite[Theorem~18(e)]{Kabganidiff}. Hence, for sufficiently small $\gamma$,
Theorem~\ref{thm:objective_prox_wellposed} ensures nonemptiness, compactness, and local uniform level-boundedness of the tensor subproblems. 
For $y\neq\ov x$, Proposition~\ref{prop:tensor_basic_comparison} and \eqref{eq:p_calm_tee} ensure 
\[
\Model^{p,q}_\gamma(\ov x,y) \ge \gf(y) +\frac{\gamma^{-1}-L}{p}\|y-\ov x\|^p 
> \gf(\ov x)+ \left(\frac{\gamma^{-1}-L}{p}-M\right)\|y-\ov x\|^p.
\]
Since $0<\gamma<\frac{1}{L+pM}$, the last coefficient is positive, whereas $\Model^{p,q}_\gamma(\ov x,\ov x)=\gf(\ov x)$.
Thus $\Tprox{\gf}{\gamma}{p,q}(\ov x)=\{\ov x\}$, i.e., $\mathcal A_\gamma^{p,q}(\ov x)=\{0\}$, and the result follows
from Theorem~\ref{thm:active_gradient_differentiability}.
\\
\ref{thm:calm_prox_regular_tee:b} Let us define the Taylor-gradient error by
\[
    e_x(y):=\nabla_y\Omega_{x,q}^f(y)-\nabla f(y).
\]
In light of $f\in C^{q+1}(\E)$ and by restricting to a neighborhood of $\ov x$, there exists $C>0$ such that 
\begin{equation}\label{eq:taylor_gradient_local_difference}
\|e_x(y_1)-e_x(y_2)\|_* \le C R^{q-1}\|y_1-y_2\|,
\end{equation}
where $R:=\|y_1-x\|+\|y_2-x\|\le 1$. Indeed, this follows by differentiating $e_x$ with respect to $y$ and using the Taylor remainder for $\nabla^2 f$. 
Let $J_p(u):=\norm{u}^{p-2}Bu$. By the standard monotonicity
inequality for the power map \cite[Eq.~(6)]{Chipot2026Uniqueness}, there exists
$\kappa_p>0$, depending only on $p$, such that
\begin{equation}\label{eq:power_monotonicity_local}
\langle J_p(u_1)-J_p(u_2),u_1-u_2\rangle \ge \kappa_p  (\|u_1\|+\|u_2\|)^{p-2}\|u_1-u_2\|^2.
\end{equation}
Let us choose $\gamma>0$ to be sufficiently small so that Assertion~\ref{thm:calm_prox_regular_tee:a} holds and $\frac{\kappa_p}{\gamma}>C+\rho$.
Since $\Tprox{\gf}{\gamma}{p,q}(\ov x)=\{\ov x\}$, outer semicontinuity and local boundedness imply that all tensor minimizers corresponding to $x$ sufficiently close to $\ov x$
remain arbitrarily close to $\ov x$. Fix such an $x$ and let $y_i\in\Tprox{\gf}{\gamma}{p,q}(x)$, $i=1,2$. Set $u_i:=y_i-x$ and $d:=y_1-y_2$. Applying the first-order optimality condition, we come to
\[
v_i:= -e_x(y_i)-\frac1\gamma J_p(u_i)\in\partial\gf(y_i).
\]
For $x$ sufficiently close to $\ov x$, the points $y_i$, the subgradients $v_i$, and the corresponding function values lie in the neighborhood where $r$-prox-regularity applies. Applying
\eqref{eq:r_prox_regular_tee} twice, with $(y_1,y_2)$ interchanged, gives
$\langle v_1-v_2,d\rangle \ge -\rho\|d\|^r$.
Consequently,
\begin{equation}\label{eq:uniqueness_balance_tee}
\frac{1}{\gamma} \langle J_p(u_1)-J_p(u_2),d\rangle \le \rho\|d\|^r + \left|\langle e_x(y_1)-e_x(y_2),d\rangle \right|.
\end{equation}
Since $\|d\|\le R$, $r\ge p$, $p\le q+1$, and $R\le1$, it holds that
\[
\|d\|^r \le R^{r-2}\|d\|^2 \le R^{p-2}\|d\|^2,
\]
and $R^{q-1}\le R^{p-2}$.
Now, applying \eqref{eq:taylor_gradient_local_difference} and \eqref{eq:power_monotonicity_local} to the inequality \eqref{eq:uniqueness_balance_tee} yields
\[
\frac{\kappa_p}{\gamma} R^{p-2}\|d\|^2 \le (C+\rho)R^{p-2}\|d\|^2.
\]
Since $\frac{\kappa_p}{\gamma}>C+\rho$, this is possible only when $d=0$, i.e., the tensor operator is locally single-valued. Its continuity follows from outer semicontinuity and local boundedness, and Theorem~\ref{thm:active_gradient_differentiability} gives the claimed $C^1$ property of $\fgam{\gf}{p,q}{\gamma}$.
\end{proof}

\subsection{Stationarity of TEE}
In this section, we study the optimization significance of stationary points of TEE. For $p=q+1$, every $y\in\Tprox{\gf}{\gamma}{q+1,q}(x)$ with $y\neq x$ provides a strict descent direction $y-x$ for TEE, without requiring differentiability. This yields an exact correspondence between directional stationarity and fixed points of TOP. We then show that this correspondence can fail when $q<p<q+1$, and we derive quantitative stationarity estimates when $q$ is even.

\begin{theorem}[Directional stationarity and tensor fixed points]\label{thm:stationary_fixed}
Assume that $f\in C^{q+1}(\E)\cap\maj{q}(q+1,L)$, $0<\gamma<L^{-1}$, and suppose that the assumptions of Lemma~\ref{lem:local_compact_reduction} hold around $x$.
Then, for every $y\in\Tprox{\gf}{\gamma}{q+1,q}(x)$,
\begin{equation}\label{eq:directional_tensor_descent}
\bigl(\fgam{\gf}{q+1,q}{\gamma}\bigr)'(x;y-x) \leq -(\gamma^{-1}-L)\norm{y-x}^{q+1}.
\end{equation}
Consequently, the following statements are equivalent:
\begin{enumerate}[label=(\textbf{\alph*}), font=\normalfont\bfseries, leftmargin=0.7cm]
\item\label{thm:stationary_fixed:a} $\bigl(\fgam{\gf}{q+1,q}{\gamma}\bigr)'(x;d)\ge 0$ for every $d\in\E$;

\item\label{thm:stationary_fixed:b}  $\Tprox{\gf}{\gamma}{q+1,q}(x)=\{x\}$;

\item\label{thm:stationary_fixed:c} $\fgam{\gf}{q+1,q}{\gamma}$ is Fr\'echet differentiable at $x$ and $\nabla\fgam{\gf}{q+1,q}{\gamma}(x)=0$.
\end{enumerate}
Each condition implies $x\in\crit(\gf)$.
Moreover, whenever TEE is differentiable at $x$, for every $y\in\Tprox{\gf}{\gamma}{q+1,q}(x)$,
\begin{equation}\label{eq:tee_gradient_step}
\left\langle\nabla\fgam{\gf}{q+1,q}{\gamma}(x), y-x\right\rangle\le-(\gamma^{-1}-L)\|y-x\|^{q+1},
\end{equation}
and
\begin{equation}\label{eq:tee_gradient_step_norm}
\left\|\nabla\fgam{\gf}{q+1,q}{\gamma}(x)\right\|_*\ge(\gamma^{-1}-L)\|y-x\|^q.
\end{equation}
\end{theorem}
\begin{proof}
Set $s:=y-x$. In light of Lemma~\ref{fact:danskin_tensor}, we have
\[
\bigl(\fgam{\gf}{q+1,q}{\gamma}\bigr)'(x;s)=\mint{z\in\Tprox{\gf}{\gamma}{q+1,q}(x)}\left\langle\nabla_xS_{q,q+1,\gamma}(x,z),s\right\rangle\leq
\left\langle\nabla_xS_{q,q+1,\gamma}(x,y),s\right\rangle.
\]
Using \eqref{eq:base_gradient} and Theorem~\ref{thm:fractional-characterization}~\ref{thm:fractional-characterization:d}, we conclude
\[
\left\langle\nabla_xS_{q,q+1,\gamma}(x,y),s\right\rangle = \frac1{q!}D^{q+1}f(x)[s]^{q+1}-\gamma^{-1}\norm{s}^{q+1}\leq -(\gamma^{-1}-L)\norm{s}^{q+1},
\]
giving \eqref{eq:directional_tensor_descent}.
If \ref{thm:stationary_fixed:a} holds, \eqref{eq:directional_tensor_descent} forces every tensor minimizer to equal $x$, so \ref{thm:stationary_fixed:b}
follows. If \ref{thm:stationary_fixed:b} holds, then $\mathcal A_\gamma^{q+1,q}(x)=\{0\}$, and Theorem~\ref{thm:active_gradient_differentiability} gives
\ref{thm:stationary_fixed:c}. The implication \ref{thm:stationary_fixed:c}$\Rightarrow$\ref{thm:stationary_fixed:a} is immediate.
Moreover, \ref{thm:stationary_fixed:b} implies $x\in\Fix(\Tprox{\gf}{\gamma}{q+1,q})$, and therefore Proposition~\ref{prop:fixed_critical} gives $x\in\crit(\gf)$.

At differentiable points, the directional derivative equals the gradient pairing, i.e., \eqref{eq:directional_tensor_descent} gives
\eqref{eq:tee_gradient_step}. The dual Cauchy inequality then yields \eqref{eq:tee_gradient_step_norm}.
\end{proof}

The restriction $p=q+1$ in Theorem~\ref{thm:stationary_fixed} is essential.
\begin{example}[Failure of stationarity preservation for $q<p<q+1$]
\label{ex:fractional_stationarity_failure}
Let $\E=\R$ and consider $f(t)=t^4-t^6$ and $g(t)=(t-1)^2+t^8-12t$, with
$q=3$, $p=\frac{7}{2}$, and $\gamma=\frac{1}{4}$. Then $\gf=f+g$ is smooth and coercive. In addition, we have
\[
f(x+h)-\Omega^f_{x,3}(x+h) = h^4\left(1-15\left(x+\frac h5\right)^2-\frac25h^2\right).
\]
As a consequence, Definition~\ref{def:maj} yields
\[
L_{7/2,3}(f)= \frac{7}{2}\supt{s>0}s^{1/2}\left(1-\frac25s^2\right)_+ = \frac{14}{5\,2^{1/4}}<4=\gamma^{-1}.
\]
At $x=0$, $\Omega^f_{0,3}(y)=0$, and 
$\Model^{7/2,3}_\gamma(0,y) = (y-1)^2+y^8-12y+\frac87|y|^{7/2}$.
Its second derivative is $2+56y^6+10|y|^{3/2}>0$, while its first derivative vanishes at $y=1$. Thus
$\Tprox{\gf}{\gamma}{7/2,3}(0)=\{1\}$. In light of Corollary~\ref{cor:bounded_below_wellposed} and
Theorem~\ref{thm:active_gradient_differentiability}, TEE is differentiable at zero. Using \eqref{eq:active_gradient_identity},
\[
\frac{d}{dx}\fgam{\gf}{7/2,3}{\gamma}(0) = \frac{f^{(4)}(0)}{3!}-\gamma^{-1} = 4-4 = 0.
\]
On the other hand, $\gf'(0)=-14\neq 0$. Thus, the stationarity of TEE does not necessarily imply the stationarity of the original objective when $q<p<q+1$, even under smoothness, coercivity, strict majorization and uniqueness of $\Tprox{\gf}{\gamma}{p,q}$ at the reference point. 
\end{example}

For even $q$, the parity structure in Theorem~\ref{thm:diagonal_parity} yields a quantitative relation between the envelope gradient and the original objective's stationarity.

\begin{corollary}[Envelope-gradient stationarity certificate]
\label{cor:gradient_stationarity_certificate}
Let $q$ be even and suppose that $f\in C^{q+1}(\E)\cap\maj{q}(q+1,L)$ and
$0<\gamma<L^{-1}$, and that the assumptions of Lemma~\ref{lem:local_compact_reduction} hold around $x$. Let TEE be differentiable at $x$, and let $y\in\Tprox{\gf}{\gamma}{q+1,q}(x)$. Then
\begin{equation}\label{eq:gradient_step_two_sided}
(\gamma^{-1}-L)\norm{y-x}^q \leq \left\|\nabla\fgam{\gf}{q+1,q}{\gamma}(x)\right\|_* \leq (\gamma^{-1}+L)\norm{y-x}^q.
\end{equation}
Moreover,
\begin{equation}\label{eq:stationarity_step_certificate}
\dist\bigl(0,\partial\gf(y)\bigr) \leq (\gamma^{-1}+L)\norm{y-x}^q, 
\end{equation}
and consequently
\begin{equation}\label{eq:stationarity_gradient_certificate}
\dist\bigl(0,\partial\gf(y)\bigr) \leq \frac{\gamma^{-1}+L}{\gamma^{-1}-L}\left\|\nabla\fgam{\gf}{q+1,q}{\gamma}(x)\right\|_*.
\end{equation}
\end{corollary}
\begin{proof}
The lower bound in \eqref{eq:gradient_step_two_sided} follows from Theorem~\ref{thm:stationary_fixed}. Since $q$ is even, Theorem~\ref{thm:diagonal_parity} gives
$\operatorname{Lip}(D^qf)\leq q!L$. Since $f\in C^{q+1}(\E)$, this implies $\norm{D^{q+1}f(x)}\leq q!L$. Using \eqref{eq:active_gradient_identity} therefore gives the upper
bound in \eqref{eq:gradient_step_two_sided}. The optimality condition of the tensor subproblem yields $v\in\widehat{\partial}g(y)\subseteq\partial g(y)$ such that
\[
0 = \nabla_y\Omega^f_{x,q}(y) +\frac1\gamma\norm{y-x}^{q-1}B(y-x)+v.
\]
In addition, in light of Theorem~\ref{thm:diagonal_parity}, we have
\[
\norm{\nabla f(y)-\nabla_y\Omega^f_{x,q}(y)}_*\leq L\norm{y-x}^q,
\]
i.e.,
\[
\norm{\nabla f(y)+v}_* \leq (\gamma^{-1}+L)\norm{y-x}^q,
\]
leading to \eqref{eq:stationarity_step_certificate}. Combining this  with the lower bound in \eqref{eq:gradient_step_two_sided} gives
\eqref{eq:stationarity_gradient_certificate}.
\end{proof}

Step-size stationarity estimates for general Taylor-like models are well established; see, e.g., \cite{DrusvyatskiyIoffeLewis2021}. The preceding corollary uses the tensor-majorant condition for
$p=q+1$ with even $q$ to relate this stationarity measure directly to the gradient of TEE.
If $p=q+1$ and TEE is differentiable at $x$, we define 
\begin{equation}\label{eq:nonlinear_tensor_scaling}
\mathcal G_x(s):= \frac1\gamma\norm{s}^{q-1}Bs - \frac1{q!}D^{q+1}f(x)[\,\cdot\,,s,\ldots,s].
\end{equation}
Consequently, for every $y\in\Tprox{\gf}{\gamma}{q+1,q}(x)$, it can be concluded that
\begin{equation}\label{eq:implicit_tensor_gradient_step}
\mathcal G_x(y-x) = -\nabla\fgam{\gf}{q+1,q}{\gamma}(x).
\end{equation}
This is an implicit nonlinear gradient relation, and we do not assert the invertibility of $\mathcal G_x$. For $q=1$, it reduces to the familiar scaled-gradient relation for the forward--backward envelope; see \cite{Stella17,kabgani2025fundamental}.
\subsection{Generic tensor method: Envelope-based analysis}
Let us consider the generic tensor method, for which the iteration is given by
\begin{equation}\label{eq:exact_tensor_iteration}
    x^{k+1}\in\Tprox{\gf}{\gamma}{p,q}(x^k).
\end{equation}
The comparison inequalities established above show that TEE provides a natural Lyapunov function for this iteration.

\begin{proposition}[Envelope decrease]\label{prop:tensor_envelope_lyapunov_general}
Let $p>1$ and $q\ge1$. Suppose that $f\in\maj{q}(p,L)$ and let $0<\gamma<L^{-1}$. Then, for every $x\in\E$ and $y\in\Tprox{\gf}{\gamma}{p,q}(x)$,
\begin{equation}\label{eq:tensor_envelope_lyapunov_general}
\fgam{\gf}{p,q}{\gamma}(y)+\frac{\gamma^{-1}-L}{p}\|y-x\|^p \le \fgam{\gf}{p,q}{\gamma}(x).
\end{equation}
\end{proposition}
\begin{proof}
Invoking Proposition~\ref{prop:tensor_basic_comparison} leads to $\fgam{\gf}{p,q}{\gamma}(x) \ge \gf(y)+\frac{\gamma^{-1}-L}{p}\|y-x\|^p$. Since $\fgam{\gf}{p,q}{\gamma}(y)\leq\gf(y)$, the result follows.
\end{proof}

\begin{theorem}[Asymptotic properties of the exact tensor iteration]\label{thm:exact_tensor_iteration_majorant}
Let $p>1$ and $q\ge1$. Suppose that $f\in\maj{q}(p,L)$, that $\gf$ is bounded below, and that $0<\gamma<L^{-1}$. Let $\{x^k\}_{k\in\Nz}$ be generated by
\eqref{eq:exact_tensor_iteration}. Then 
\begin{enumerate}[label=(\textbf{\alph*}), font=\normalfont\bfseries, leftmargin=0.7cm]
\item\label{thm:exact_tensor_iteration_majorant:a} the iteration is well defined and
\begin{equation}\label{eq:tensor_step_summability_general}
\sum_{k=0}^{\infty}\|x^{k+1}-x^k\|^p \le \frac{p}{\gamma^{-1}-L}\,\left(\fgam{\gf}{p,q}{\gamma}(x^0)-\inf\gf\right)<+\infty.
\end{equation}
In particular, $\|x^{k+1}-x^k\|\to 0$;

\item\label{thm:exact_tensor_iteration_majorant:b} $\fgam{\gf}{p,q}{\gamma}(x^{k+1}) \leq \gf(x^{k+1}) \leq \fgam{\gf}{p,q}{\gamma}(x^k)\leq\gf(x^k)$;

\item\label{thm:exact_tensor_iteration_majorant:c} there exists $\ell\in\R$ such that $\fgam{\gf}{p,q}{\gamma}(x^k)\to\ell$ and $\gf(x^k)\to\ell$;

\item\label{thm:exact_tensor_iteration_majorant:d} every cluster point $\ov x$ of $\{x^k\}_{k\in\Nz}$ satisfies
\[
\ov x\in \Fix(\Tprox{\gf}{\gamma}{p,q}) \subseteq \crit(\gf),
\]
and $\gf(\ov x) = \fgam{\gf}{p,q}{\gamma}(\ov x) = \ell$.
\end{enumerate}
\end{theorem}
\begin{proof}
\ref{thm:exact_tensor_iteration_majorant:a} Since $\gf$ is bounded below, Corollary~\ref{cor:bounded_below_wellposed} ensures that \eqref{eq:exact_tensor_iteration} is well defined.
In light of Proposition~\ref{prop:tensor_envelope_lyapunov_general}, we have
\[
\fgam{\gf}{p,q}{\gamma}(x^{k+1})+\frac{\gamma^{-1}-L}{p}\norm{x^{k+1}-x^k}^p \leq \fgam{\gf}{p,q}{\gamma}(x^k).
\]
Summing up both sides of this inequality from $k=0$ to $N-1$ and using $\fgam{\gf}{p,q}{\gamma}\geq\inf_{\E}\gf$, it is evident that
\[
\frac{\gamma^{-1}-L}{p} \sum_{k=0}^{N-1}\norm{x^{k+1}-x^k}^p \leq \fgam{\gf}{p,q}{\gamma}(x^0)-\inf\gf.
\]
Letting $N\to+\infty$ proves \eqref{eq:tensor_step_summability_general}, and hence $x^{k+1}-x^k\to 0$.
\\
\ref{thm:exact_tensor_iteration_majorant:b} It follows from Proposition~\ref{prop:tensor_basic_comparison}~\ref{prop:tensor_basic_comparison:b} and \ref{prop:tensor_basic_comparison:c}.
\\
\ref{thm:exact_tensor_iteration_majorant:c}
Since $\{\fgam{\gf}{p,q}{\gamma}(x^k)\}_{k\in\Nz}$ is nonincreasing and bounded below, it converges to some $\ell\in\R$, and Assertion~\ref{thm:exact_tensor_iteration_majorant:b} yields $\gf(x^k)\to\ell$. 
\\
\ref{thm:exact_tensor_iteration_majorant:d}
Let $x^{k_j}\to\ov x$. Since $x^{k_j+1}-x^{k_j}\to 0$, $x^{k_j+1}\to\ov x$. Theorem~\ref{thm:objective_prox_wellposed}~\ref{thm:objective_prox_wellposed:f} gives
$\ov x \in\Tprox{\gf}{\gamma}{p,q}(\ov x)$, and Proposition~\ref{prop:fixed_critical} implies
$\ov x\in\crit(\gf)$. Finally, the continuity of TEE gives $\fgam{\gf}{p,q}{\gamma}(\ov x)=\ell$, while the fixed-point identity gives 
$\gf(\ov x) = \fgam{\gf}{p,q}{\gamma}(\ov x)$.
\end{proof}

Next, we derive a stationarity estimate without imposing a global H\"older condition on $D^qf$.
For $\epsilon>0$, call $x\in\dom\gf$ an $\epsilon$-stationary point of $\gf$ if
$\dist(0,\partial\gf(x))\leq\epsilon$.
For a sequence $\{x^k\}_{k\in\Nz}$, define
\[
N_\epsilon := \min\left\{ N\geq1: \mint{0\leq k<N} \dist\bigl(0,\partial\gf(x^{k+1})\bigr)\leq\epsilon
\right\}.
\]

\begin{theorem}[Stationarity iteration complexity on a bounded trajectory]\label{thm:bounded_sequence_complexity}
Let $q<p\le q+1$. Suppose that $f\in C^q(\E)\cap\maj{q}(p,L)$, $\inf_\E\gf> -\infty$, and $0<\gamma<L^{-1}$. Let $\{x^k\}_{k\in\Nz}$ be generated by \eqref{eq:exact_tensor_iteration} and suppose that it is bounded. Let $K\subset\E$ be a compact convex set containing all iterates, and suppose that there exists $C_K<+\infty$ such that
\begin{equation}\label{eq:local_gradient_model_accuracy}
\|\nabla f(y)-\nabla_y\Omega^f_{x,q}(y)\|_*\le C_K\|y-x\|^{p-1} \quad x,y\in K.
\end{equation}
Set $b_K:=\gamma^{-1}+C_K$. Then, for every $k\in\Nz$,
\begin{equation}\label{eq:local_stationarity_each_iteration}
\dist\bigl(0,\partial\gf(x^{k+1})\bigr)\le b_K\|x^{k+1}-x^k\|^{p-1}.
\end{equation}
Moreover, for every integer $N\ge1$,
\begin{equation}\label{eq:local_stationarity_complexity}
\mint{0\le k<N}\dist\bigl(0,\partial\gf(x^{k+1})\bigr)\le b_K
\left(\frac{p(\fgam{\gf}{p,q}{\gamma}(x^0)-\inf_\E\gf)}{(\gamma^{-1}-L)N}\right)^{\frac{p-1}{p}},
\end{equation}
and
\[
N_\epsilon\le\max\left\{1,\left\lceil\frac{p(\fgam{\gf}{p,q}{\gamma}(x^0)-\inf_\E\gf)}{\gamma^{-1}-L}\left(\frac{b_K}{\epsilon}\right)^{\frac{p}{p-1}}\right\rceil\right\}.
\]
Consequently, $N_\epsilon=\mathcal O\!\left(\epsilon^{-p/(p-1)}\right)$.
\end{theorem}
\begin{proof}
For each $k$, optimality of $x^{k+1}$ gives $v^{k+1}\in\widehat\partial g(x^{k+1})\subseteq\partial g(x^{k+1})$ such that
\[
0=\nabla_y\Omega^f_{x^k,q}(x^{k+1})+\gamma^{-1}\|x^{k+1}-x^k\|^{p-2}B(x^{k+1}-x^k)+v^{k+1}.
\]
Since $\nabla f(x^{k+1})+v^{k+1}\in\partial\gf(x^{k+1})$, \eqref{eq:local_gradient_model_accuracy} gives \eqref{eq:local_stationarity_each_iteration}. Theorem~\ref{thm:exact_tensor_iteration_majorant} yields
\[
\mint{0\le k<N}\|x^{k+1}-x^k\|^p\le\frac{p(\fgam{\gf}{p,q}{\gamma}(x^0)-\inf_\E\gf)}{(\gamma^{-1}-L)N},
\]
and the remaining assertions follow.
\end{proof}

\begin{remark}
If $f$ is $C^{q+1}$ on a neighborhood of $K$, then \eqref{eq:local_gradient_model_accuracy} holds with
\[
C_K=\frac{H_K}{q!}D_K^{q-p+1},\quad H_K:=\max_{z\in K}\|D^{q+1}f(z)\|,
\quad D_K:=\max\{1,\operatorname{diam}K\}.
\]
If $\gf$ is level-bounded, the bounded-trajectory assumption follows automatically from the interlacing descent in Theorem~\ref{thm:exact_tensor_iteration_majorant}. Thus, the theorem does
not require global H\"older or Lipschitz continuity of $D^qf$. The global descent follows from the one-sided condition $f\in\maj{q}(p,L)$, while the stationarity estimate requires only gradient-model accuracy on the region visited by the trajectory.
\end{remark}

For comparison with standard global smoothness assumptions, we record the following estimate.

\begin{lemma}[H\"older Taylor-gradient accuracy]\label{lem:holder_gradient_model_accuracy}
Let $q\ge1$ and $\nu\in(0,1]$. Suppose that $f\in C^q(\E)$ and $D^qf$ is globally $\nu$-H\"older continuous with constant $H_q\ge0$. Set
$L_{q,\nu}:=H_q\frac{\Gamma(\nu+1)}{\Gamma(q+\nu)}$. Then $f\in\maj{q}(q+\nu,L_{q,\nu})$, and for all $x,y\in\E$,
\begin{equation}\label{eq:holder_gradient_model_accuracy}
\|\nabla f(y)-\nabla_y\Omega_{x,q}^f(y)\|_* \le L_{q,\nu}\|y-x\|^{q+\nu-1}.
\end{equation}
\end{lemma}
\begin{proof}
The membership $f\in\maj{q}(q+\nu,L_{q,\nu})$ is Fact~\ref{lem:holder_sufficient}. For $q=1$, \eqref{eq:holder_gradient_model_accuracy} is exactly the $\nu$-H\"older continuity of $\nabla f$. Assume $q\ge2$ and set $h=y-x$. Applying the integral Taylor remainder formula to the map $\nabla f$ leads to
\begin{align*}
 &\nabla f(x+h)-\nabla_y\Omega_{x,q}^f(x+h)
 \\&\quad= \frac1{(q-2)!}\int_0^1(1-t)^{q-2}\bigl(D^qf(x+th)-D^qf(x)\bigr)[\,\cdot\,,h,\ldots,h]dt,
\end{align*}
i.e.,
\[
\|\nabla f(x+h)-\nabla_y\Omega_{x,q}^f(x+h)\|_* \le \frac{H_q}{(q-2)!} \int_0^1(1-t)^{q-2}t^\nu dt\,\|h\|^{q+\nu-1},
\]
i.e., \eqref{eq:beta-gamma} reduces the coefficient to $H_q\Gamma(\nu+1)/\Gamma(q+\nu)=L_{q,\nu}$, proving the claim.
\end{proof}

\begin{corollary}[Complexity under global H\"older regularity]\label{thm:matched_holder_tensor_complexity}
Let $q\ge1$ and $\nu\in(0,1]$, and set $p:=q+\nu$ and $L_{q,\nu}:=H_q\frac{\Gamma(\nu+1)}{\Gamma(q+\nu)}$.
Suppose that $f\in C^q(\E)$, that $D^qf$ is globally $\nu$-H\"older continuous with constant $H_q\ge0$, that $\gf$ is bounded below, and that $0<\gamma<L_{q,\nu}^{-1}$. Define $b_{\gamma,\nu}:=\gamma^{-1}+L_{q,\nu}$. Let $\{x^k\}_{k\in\Nz}$ be generated by
\eqref{eq:exact_tensor_iteration}. Then, for every integer $N\geq1$,
\begin{equation}\label{eq:matched_holder_stationarity_complexity}
\mint{0\le k<N}\dist\bigl(0,\partial\gf(x^{k+1})\bigr) \le b_{\gamma,\nu}
\left(\frac{p(\fgam{\gf}{p,q}{\gamma}(x^0)-\inf\gf)}{(\gamma^{-1}-L_{q,\nu})N}\right)^{\frac{p-1}{p}}.
\end{equation}
Moreover, 
\[
N_\epsilon \leq \max\left\{1,\left\lceil \frac{p(\fgam{\gf}{p,q}{\gamma}(x^0)-\inf\gf)}{\gamma^{-1}-L_{q,\nu}}\left(\frac{b_{\gamma,\nu}}{\epsilon}\right)^{\frac{p}{p-1}}\right\rceil\right\},
\]
and therefore $N_\epsilon=\mathcal O\left(\epsilon^{-\frac{q+\nu}{q+\nu-1}}\right)$.
\end{corollary}
\begin{proof}
Using the optimality of $x^{k+1}$ and invoking Lemma~\ref{lem:holder_gradient_model_accuracy} imply
\[
\dist\bigl(0,\partial\gf(x^{k+1})\bigr) \leq b_{\gamma,\nu}\norm{x^{k+1}-x^k}^{p-1}.
\]
The conclusion follows from Theorem~\ref{thm:exact_tensor_iteration_majorant}.
\end{proof}

\begin{remark}[The case $p=q+1$]\label{rem:matched_holder_diagonal_endpoint}
Suppose that $\gf$ is bounded below. If $q$ is even and $f\in C^{q+1}(\E)\cap\maj{q}(q+1,L)$, then Theorem~\ref{thm:diagonal_parity} gives $\operatorname{Lip}(D^qf)\leq q!L$. Hence, the preceding corollary applies with $H_q=q!L$, $L_{q,1}=L$, and $b_{\gamma,1}=\gamma^{-1}+L$, and yields $N_\epsilon = \mathcal O\left(\epsilon^{-\frac{q+1}{q}}\right)$. If $q$ is odd, membership in $\maj{q}(q+1,L)$ generally does not provide a global bound on $\norm{D^{q+1}f}$. Nevertheless,
Theorem~\ref{thm:bounded_sequence_complexity} applies whenever the generated sequence is bounded.
\end{remark}

\begin{remark}[Connection with high-order regularization complexity]
\label{rem:classical_complexity_connection}
For $p=q+\nu$, the order $\mathcal O\left(\epsilon^{-\frac{q+\nu}{q+\nu-1}}\right)$ coincides with the standard first-order stationarity exponent for regularization methods based on a $q$th-order Taylor model under $\nu$-H\"older continuity of $D^qf$; see, e.g., \cite{CartisGouldToint2019Universal,CartisGouldToint2020Sharp}. The above bound is an iteration-complexity result for exact global solutions of fixed-parameter tensor subproblems. It does not count unsuccessful trials, parameter updates, derivative evaluations, or the cost of solving the subproblems, and therefore should not be interpreted as an evaluation-complexity bound. For $p>q+1$, Theorem~\ref{thm:exact_tensor_iteration_majorant} still applies whenever $f\in\maj{q}(p,L)$. However, Theorem~\ref{thm:superdiagonal} shows that finite tensor majorization is highly restrictive in this regime. The $q$th-order Taylor polynomial must already globally majorize $f$, and the optimal majorization constant is zero.
\end{remark}

\section{Conclusion} \label{sec:conclusion}
We developed a tensor envelope (TEE) for composite optimization based on a regularized $q$th-order Taylor model of the smooth term. We first studied the associated tensor descent lemma and characterized its behavior in the regimes $q<p<q+1$, $p=q+1$, and $p>q+1$. Specifically, the first two regimes are governed by one-sided H\"older-type information on $D^qf$, whereas for $p>q+1$ a finite majorization exists only if the $q$th-order Taylor model already majorizes $f$ globally. We then established the fundamental variational properties of the tensor operator and envelope, including well-definedness, continuity, preservation of global minimizers, level-boundedness, and
differentiability of TEE. For $p=q+1$, we characterized directional stationarity of TEE through singleton tensor fixed points and related these points to critical points of the original objective.
Finally, using TEE as a Lyapunov function for the exact tensor iteration, we established asymptotic regularity, convergence of the objective and envelope values, and criticality of cluster points. For bounded trajectories with $q<p\le q+1$, we further showed that an $\epsilon$-stationary point is reached within $\mathcal{O}(\epsilon^{-p/(p-1)})$ iterations. Notably, neither the tensor descent lemma nor this complexity bound requires global Lipschitz or H\"older continuity of $D^qf$. The theoretical results presented here provide a foundation for future numerical developments, which we pursue in ongoing work \cite{kabgani2026boosted}.

\newpage
\bibliographystyle{siamplain}
\bibliography{references}
\end{document}